\documentclass[12pt]{amsart}
\usepackage{amsmath,amssymb,amsthm,xspace,a4wide,amscd,eucal}
\usepackage[all]{xy}

\usepackage{mathptmx}       
\usepackage{helvet}         
\usepackage{courier}        
\usepackage{type1cm}        
\usepackage{graphicx}        
\usepackage{multicol}        
\usepackage[bottom]{footmisc}

\makeindex             

\usepackage{float}
\usepackage{graphics}
\usepackage{amssymb,amsmath,amscd}
\usepackage{mathrsfs}
\usepackage{enumerate}
\usepackage[all]{xy}
\usepackage{verbatim}
\usepackage{color}
\usepackage[vcentermath,enableskew]{youngtab}

\newcommand{\C}{{\mathbb C}}

\newcommand{\Z}{{\mathbb Z}}

\newcommand{\nc}{\newcommand}

\nc{\one}{\mbox{\bf 1}}

\nc{\invtensor}{\underset{\leftarrow}{\otimes}}

\nc{\ad}{\operatorname{ad}}

\nc{\tr}{\operatorname{tr}}
\nc{\str}{\operatorname{str}}

\nc{\rank}{\operatorname{rank}}
\nc{\rnk}{\operatorname{rank}}

\nc{\corank}{\operatorname{corank}}

\nc{\Sym}{\operatorname{Sym}}

\nc{\sym}{\operatorname{sym}}

\nc{\id}{\operatorname{id}}
\nc{\Id}{\operatorname{Id}}

\nc{\htt}{\operatorname{ht}}

\nc{\Norm}{\operatorname{Norm}}
\nc{\Ker}{\operatorname{Ker}}
\nc{\rker}{\operatorname{rKer}}

\nc{\im}{\operatorname{Im}}

\nc{\osp}{\operatorname{osp}}
\nc{\sgn}{\operatorname{sgn}}
\nc{\Mod}{\operatorname{Mod}}
\nc{\Mat}{\operatorname{Mat}}
\nc{\sMat}{\operatorname{sMat}}
\nc{\Soc}{\operatorname{Soc}}
\nc{\Inj}{\operatorname{Inj}}
\nc{\Hom}{\operatorname{Hom}}
\nc{\End}{\operatorname{End}}
\nc{\supp}{\operatorname{supp}}
\nc{\Card}{\operatorname{Card}}
\nc{\Ann}{\operatorname{Ann}}
\nc{\Ind}{\operatorname{Ind}}
\nc{\Coind}{\operatorname{Coind}}
\nc{\wt}{\operatorname{wt}}
\nc{\spn}{\operatorname{span}}
\nc{\ch}{\operatorname{ch}}
\nc{\codim}{\operatorname{codim}}
\nc{\Stab}{\operatorname{Stab}}
\nc{\Supp}{\operatorname{Supp}}

\nc{\Irr}{\operatorname{Irr}}

\nc{\Spec}{\operatorname{Spec}}

\nc{\Prim}{\operatorname{Prim}}

\nc{\Aut}{\operatorname{Aut}}

\nc{\Fract}{\operatorname{Fract}}

\nc{\gr}{\operatorname{gr}}

\nc{\HC}{\operatorname{HC}}
\nc{\Ad}{\operatorname{Ad}}

\nc{\wdchi}{\widetilde{\chi}}

\nc{\wdH}{\widetilde{H}}

\nc{\wdN}{\widetilde{N}}

\nc{\wdM}{\widetilde{M}}

\nc{\wdO}{\widetilde{O}}

\nc{\wdR}{\widetilde{R}}

\nc{\wdS}{\widetilde{S}}

\nc{\wdV}{\widetilde{V}}

\nc{\wdC}{\widetilde{C}}

\nc{\soc}{\operatorname{soc}}

\nc{\cN}{\operatorname{\mathcal M}^{\#}}
\nc{\cM}{\operatorname{\mathcal M}}
\nc{\Ob}{\operatorname{\mathcal Ob}}

\nc{\cO}{\operatorname{\mathcal O}}
\nc{\cC}{\operatorname{\mathcal C}}
\nc{\cD}{\operatorname{\mathcal D}}
\nc{\Dglie}{\operatorname{{\mathcal D}glie}}

\nc{\Fin}{\operatorname{{\mathcal F}in}}

\nc{\Sg}{{\mathcal S}({\mathfrak g})}

\nc{\Ug}{{\mathcal U}({\mathfrak g})}

\nc{\Zg}{{\mathcal Z}({\mathfrak g})}

\nc{\tZg}{{\widetilde{\mathcal Z}({\mathfrak g})}}

\nc{\Zk}{{\mathcal Z}({\mathfrak k})}

\nc{\Sh}{{\mathcal S}({\mathfrak h}_0)}

\nc{\Uh}{{\mathcal U}({\mathfrak h})}

\nc{\Up}{{\mathcal U}({\mathfrak p})}

\nc{\Ub}{{\mathcal U}({\mathfrak b})}

\nc{\Zh}{{\mathcal Z}({\mathfrak h})}
\nc{\Ah}{{\mathcal A}({\mathfrak h})}

\nc{\Ag}{{\mathcal A}({\mathfrak g})}

\nc{\Ap}{{\mathcal A}({\mathfrak p})}

\nc{\Zp}{{\mathcal Z}({\mathfrak p})}

\nc{\cZ}{\mathcal Z}

\nc{\cS}{\mathcal S}

\nc{\cB}{\mathcal B}
\nc{\cP}{\mathcal P}

\nc{\cA}{\mathcal A}

\nc{\cU}{\mathcal U}

\nc{\cH}{\mathcal H}

\nc{\cL}{\mathcal L}

\nc{\cF}{\mathcal F}

\nc{\fg}{\mathfrak g}

\nc{\CO}{\mathcal O}

\nc{\Cl}{\mathcal {C}\ell}

\nc{\fn}{\mathfrak n}

\nc{\fm}{\mathfrak m}

\nc{\fp}{\mathfrak p}

\nc{\fh}{\mathfrak h}

\nc{\ft}{\mathfrak t}

\nc{\fk}{\mathfrak k}

\nc{\fb}{\mathfrak b}

\nc{\fI}{\mathfrak I}

\nc{\veps}{\varepsilon}
\nc{\vareps}{\varepsilon}
\def\ga{\mathfrak{a}}
\def\gb{\mathfrak{b}}

\def\gg{\mathfrak{g}}
\def\gh{\mathfrak{h}}

\def\gl{\mathfrak{l}}

\def\go{\mathfrak{o}}
\def\gp{\mathfrak{p}}

\def\gr{\mathfrak{r}}
\def\gs{\mathfrak{s}}

\def\cA{\mathcal{A}}
\def\cB{\mathcal{B}}
\def\cC{\mathcal{C}}
\def\cD{\mathcal{D}}

\def\cF{\mathcal{F}}

\def\cH{\mathcal{H}}

\def\cL{\mathcal{L}}
\def\cM{\mathcal{M}}
\def\cN{\mathcal{N}}
\def\cO{\mathcal{O}}
\def\cP{\mathcal{P}}

\def\cS{\mathcal{S}}

\def\cU{\mathcal{U}}

\def\cZ{\mathcal{Z}}

\def\Z{\mathbb{Z}}

\def\gh{\mathfrak h}

\def\im{{\rm im}}

\newtheorem{thm}{Theorem}[section]
\newtheorem{lem}[thm]{Lemma}
\newtheorem{prop}[thm]{Proposition}
\newtheorem{cor}[thm]{Corollary}
\newtheorem{defn}[thm]{Definition}
\newtheorem{examp}[thm]{Example}

\newtheorem{rmk}[thm]{Remark}

\begin{document}

\title{Simple bounded weight modules of $\displaystyle{\mathfrak{sl}(\infty)}$, $\displaystyle{\mathfrak{o}(\infty)}$, $\displaystyle{\mathfrak{sp}(\infty)}$}

\author[Dimitar Grantcharov]{\;Dimitar Grantcharov}

\address{
Dimitar Grantcharov
\newline Department of Mathematics
\newline University of Texas, Arlington
\newline Arlington, TX 76019-0408
\newline  USA}
\email{grandim@uta.edu }

\author[Ivan Penkov]{\;Ivan Penkov}

\address{
Ivan Penkov
\newline Jacobs University Bremen
\newline Campus Ring 1
\newline 28759 Bremen, Germany}
\email{i.penkov@jacobs-university.de}

\maketitle

\begin{abstract}
We classify the simple  bounded weight modules of ${\mathfrak{sl}(\infty})$, 
${\mathfrak{o}(\infty)}$ and 
${\mathfrak{sp}(\infty)}$, and compute their annihilators  in $U({\mathfrak{sl}(\infty}))$, 
$U({\mathfrak{o}(\infty))}$,
$U({\mathfrak{sp}(\infty))}$, respectively.
\end{abstract}

{\small
}

\section{Introduction}

In recent years, the theory of representations of the three simple finitary complex Lie algebras $\mathfrak{sl}(\infty)$, $\mathfrak{o}(\infty)$, $\mathfrak{sp}(\infty)$ has been developing actively, see \cite{CP}, \cite{DPS2}, \cite{HPS}, \cite{Na}, \cite{PS},  \cite{PStyr},  \cite{PZ4}. In general, this representation theory is  much richer  than the representation theory of a  simple finite-dimensional Lie algebra. Nevertheless, some problems admit a simpler answer for $\mathfrak{sl}(\infty)$, $\mathfrak{o}(\infty)$, $\mathfrak{sp}(\infty)$ than for a finite-dimensional simple Lie algebra. This applies for instance to the classification of primitive ideals in the enveloping  algebra $U(\mathfrak{sl}(\infty))$, see \cite{PP}.

In this paper we solve a classification problem which also admits a relatively simple answer compared to the finite-dimensional case: This is the problem of classifying simple bounded weight modules over the Lie algebras $\mathfrak{sl}(\infty)$, $\mathfrak{o}(\infty)$, $\mathfrak{sp}(\infty)$.

The desirability of such a classification has been clear for about 20 years. Indeed, the classification of bounded infinite-dimensional simple ${\mathfrak{sl}(n+1)}$-, ${\mathfrak{sp}(2n)}$-modules given by Mathieu in 1998 (and following earlier work of Benkart, Britten, Fernando, Futorny, Lemire,  and others) has  been  a milestone in the theory of weight modules of finite-dimensional Lie algebras. In the study of weight modules of ${\mathfrak{sl}(\infty)}$, ${\mathfrak{o}(\infty)}$, ${\mathfrak{sp}(\infty)}$, and especially of weight modules with finite weight multiplicities, a detailed understanding of the simple bounded modules is absolutely necessary.

 Soon after the celebrated work of Mathieu \cite{M}, and the work of Dimitrov and the second author \cite{DP}, Dimitrov gave several seminar talks in which he sketched a classification of simple weight $\mathfrak{sl}(\infty)$-modules with finite-dimensional weight spaces. As  this classification has still not appeared,  we consider the problem of classifying  the simple  bounded weight $\mathfrak{sl}(\infty)$-, $\mathfrak{o}(\infty)$-, $\mathfrak{sp}(\infty)$-modules from scratch. 

 Our starting point was a recollection of  Dimitrov's idea that bounded simple $\mathfrak{sl}(\infty)$-modules should be multiplicity free. This recollection turned out  to be essentially correct, and we show that all \emph{nonitegrable} simple $\mathfrak{sl}(\infty)$-  and $\mathfrak{sp}(\infty)$-modules are multiplicity free.

A brief account of the contents of the paper is as follows. In Section 2 we have collected  all necessary results on weight modules of finite-dimensional Lie algebras. This section is based on work  of Fernando, Mathieu and others, but also contains some technical results for which we found no reference. Section 3 is a summary of structural properties  of the Lie algebras $\mathfrak{sl}(\infty)$, $\mathfrak{o}(\infty)$, $\mathfrak{sp}(\infty)$. The main results of the paper are spread over Sections 4--7. Section 4 is devoted to general results on bounded weight  $\mathfrak{sl}(\infty)$-,  $\mathfrak{0}(\infty)$- ,  $\mathfrak{sp}(\infty)$-modules. Integrable  bounded weight modules are classified in Section 5, and nonintegrable bounded $\mathfrak{sl}(\infty)$- and $\mathfrak{sp}(\infty)$-modules are classified in Section 6. Finally, in Section 7 the primitive ideals arising from bounded weight modules are computed.

{\bf Acknowledgment.} DG was supported in part by Simons Collaboration Grant 358245. IP was supported in part by DFG grants PE 980/6-1 and 980/7-1. The results in Subsection \ref{subsec-int-sim-sl} which concern the modules $\Lambda_A^{\infty}V$ have been established by the second author together with Aban Husain while working on her Bachelor's Thesis at Jacobs University. We also thank Lucas Calixto and Todor Milev for useful discussions.

\section{Background on weight modules of $\mathfrak{sl}(n+1)$, ${\mathfrak{o}(2n+1)}$, ${\mathfrak{o}(2n)}$, ${\mathfrak{sp}(2n)}$.} \label{sec-background}

\subsection{Notation} \label{subsec-notation}
In this paper the ground field is $\mathbb C$. 
All vector spaces, algebras, and tensor
products are assumed to be over $\mathbb C$ unless otherwise
stated. Upper star  $\cdot^*$  indicates dual space. We write $\langle \:\: \rangle_A$ for span over a monoid $A$. By $\C^{{\mathbb Z}_{>0}}$ we denote the space of all infinite sequences ${\bf a} = (a_1,a_2,...)$ of complex numbers, and by ${\C}^{{\mathbb Z}_{>0}}_{\rm f}$ the set of all finite sequences. Similarly we define $\Z^{{\mathbb Z}_{>0}}_{\rm f}$ and ${{\mathbb Z}_{>0}}^{{\mathbb Z}_{>0}}_{\rm f}$. 
For a finite or infinite sequence ${\mathbf a} = (a_1, a_2,...)$ of complex numbers, by ${\rm Int} ({\bf a})$, respectively, by ${\rm Int}^+  ({\bf a})$ or ${\rm Int}^- ({\bf a})$,  we denote the subset of ${\mathbb Z}_{>0}$ consisting of all $i$ such that $a_i \in {\mathbb Z}$, respectively, $a_i \in {\mathbb Z}_{\geq 0}$ or $a_i \in {\mathbb Z}_{< 0}$.   If  ${\bf a}$ is  a finite sequence, we set $|{\bf a}| := \sum_{i > 0} a_i$.

For the sequences $\underbrace{(x,x,...,x)}_{\text{($n$ times)}}$ and  $(x,x,...)$ we sometimes use the short notations $x^{(n)}$ and $x^{(\infty)}$ . Sequences like $(\underbrace{x,x,...,x}_{\text{($n$ times)}}, \underbrace{y,y,...,y}_{\text{($m$ times)}})$ may be abbreviated as $(x^{(n)}, y^{(m)})$. For arbitrary sets $A$ and $B$, we write $A \subset B$ (respectively, $A \subsetneq B$) if $A$ is a subset (respectively, proper subset) of $B$.  By $S^\cdot(\cdot)$ and $\Lambda^\cdot(\cdot)$ we denote respectively the symmetric and exterior algebra of a vector space.

\subsection {Generalities} \label{subsec-general-sln}  
Let $\fg_n:=\mathfrak{sl}({n+1})$, $\mathfrak{o}({2n+1})$ , $\mathfrak{o}({2n})$, $\mathfrak{sp}({2n})$, and let $U_{n}=U(\mathfrak{g}_n)$ be the enveloping algebra of $\mathfrak{g}_n$. By $\mathfrak{h}_n$ we denote a fixed Cartan subalgebra of $\mathfrak{g}_n$ and $Q_n$ stands for the root lattice of $\mathfrak{g}_n$. We use Bourbaki's notation for the roots of $\mathfrak{g}_n$, $\Delta_n$ stands for the roots of $\fg_n$ with respect to $\fh_n$, and in all four cases we have well-definded vectors $\varepsilon_j$ which belong to $\mathfrak{h}_n^*$ for $\mathfrak{g}_n\neq\mathfrak{sl}({n+1})$. For $\fg_n=\mathfrak{sl}(n+1)$, the vectors $\varepsilon_j$ belong to the dual of a respective Cartan subalgebra of $\mathfrak{gl}(n+1)$. For $\fg_n=\go(2n+1),\go(2n),\gs\gp(2n)$ we identify $\fh_n^*$ with $\mathbb{C}^{n}$: $\fh_n^*\ni\lambda=(\lambda_1,\dots,\lambda_{n})=\Sigma_{i=1}^n\lambda_i\varepsilon_i$. When $\fg_n=\gs\gl(n+1)$ we use the same notation  $\lambda=(\lambda_1,\dots, \lambda_{n+1})=\sum_{i=1}^{n+1}\lambda_i\varepsilon_i$ for the weight of $\mathfrak{sl}(n+1)$ and  $\mathfrak{gl}(n+1)$, and automatically consider the projection of $\lambda$ in $\fh^*_n$ when we think of $\lambda$ as a weight of $\mathfrak{sl}(n+1)$. In this connection, note that if $Q_{\mathfrak{gl}(n+1)}$ denotes the root lattice of  $\mathfrak{gl}(n+1)$, the projection $Q_{\mathfrak{gl}(n+1)} \to Q_n$ is an isomorphism.

A {\it weight module} of $\gg_n$ is a module $M$ for which 
$$
M = \bigoplus_{\lambda \in \gh_n^*} M^{\lambda}
$$
where $M^{\lambda} = \{ m \in M \; | \; h\cdot m = \lambda(h)m, \forall \, h \in \gh_n\}$.  The\emph{ support of} a weight module $M$ is the set 
$$
\Supp M := \{\lambda \in \gh_n^* \; | \; M^{\lambda} \neq 0\}.
$$
Unless stated otherwise, we will assume that dim$M^\lambda<\infty$ for any $\lambda\in \Supp M$.

The Lie algebra $\fg_n$ has a \emph{ natural representation}, denoted respectively by $V_{n+1}$, $V_{2n+1}$, $V_{2n}$ and $V_{2n}$ for $\fg_n=\gs\gl(n+1)$,   $\go(2n+1)$, $\go(2n)$ and $\gs\gp(2n)$. A natural representation is characterized, up to isomorphism, by its support:

$$ \Supp V_{n+1} =  \{\varepsilon_i \, |1\leq i\leq n+1 \} \text{ for} \,\, \fg=\mathfrak{sl}(n+1),$$
$$   \Supp V_{2n+1} = \{0,\pm\varepsilon_i  \, |1\leq i\leq n \} \text{ for} \,\, \fg=\mathfrak{o}(2n+1),$$ 
$$   \Supp V_{2n} = \{\pm\varepsilon_i  \, |1\leq i\leq n \} \text{ for} \,\, \fg=\mathfrak{o}(2n), \mathfrak{sp}(2n).$$

Let $M$ be a $\fg_n$-module.  We say that a root space $\gg_n^{\alpha}$ acts {\it locally finitely} (respectively, {\it injectively}) on  $M$ if  $\gg_n^{\alpha}$ acts {locally finitely} (respectively, {\it injectively}) on every $m$ in $M$. A weight $\gg_n$-module $M$ will be called  \emph{uniform}  if every root space of $\gg_n$ acts either locally finitely or injectively on $M$. A simple weight module $M$ is always uniform. Indeed, let $\alpha \in \Delta$ and  $0 \neq x \in \gg_n^{\alpha}$. Then the set $M(\alpha)$ of all vectors $m \in M$ annihilated by a (variable) power of $x$ is a $\gg_n$-submodule. If $M(\alpha) = 0$, then $x$ acts injectively. Otherwise $M(\alpha) = M$ and $x$ acts locally finitely, cf. Lemma 3.1  in \cite{DMP}.

If $M$  is uniform,  we set $\Delta_n = \Delta_n^{M{\rm-fin}} \sqcup \Delta_n^{M{\rm-inf}}$, where $ \Delta_n^{M{\rm-fin}}$ (respectively,  $\Delta_n^{M{\rm-inf}}$) is the set of roots that act {locally finitely} (respectively, {\it injectively}) on $M$. By $C(M)$ we denote the cone $\langle \Delta_n^{M{\rm-inf}}\rangle_{\mathbb Z_{\geq 0}}$.  Note that, $\alpha,\beta\in\Delta_n^{M-inf}$, $\alpha+\beta\in\Delta_n$ implies $\alpha+\beta\in\Delta_n^{M-inf}$, see \cite{DMP}. Therefore, if $\gamma\in\Delta_n\cap C(M)$ then $\gamma\in\Delta^{M-inf}$.

A {\it bounded weight module} (or simply \emph{bounded module}) $M$ of $\gg_n$ is by definition a weight module all of whose weight multiplicities are bounded by a fixed constant $c$: $\dim M^{\lambda} < c$. A weight module is {\it cuspidal} if the root vectors of $\gg_n$ act injectively on $M$. In particular, every cuspidal module $N$  is bounded since all weight multiplicities of $N$ coincide. The {\it degree} of a bounded $\gg_n$-module $M$ is its maximal weight multiplicity:
$$
\deg M := \max \{\dim M^{\lambda} \; | \; \lambda \in \gh_n^*\}.
$$
A \emph{multiplicity-free} $\mathfrak{g}_n$-module is a bounded $\fg_n$-module of degree 1. Note that some authors, see for instance \cite{BBL}, \cite{BL1} call multiplicity-free weight modules pointed modules. 
The {\it essential support} of a bounded module $M$ is 
$$
{\rm Supp}_{\rm ess}\,M := \{ \lambda \in {\rm Supp}\, M \; | \; \dim M^{\lambda} = \deg M \}.
$$  

By ${\mathcal B} (\gg_n)$ we denote the category of bounded weight $\gg_n$-modules. 
It is a theorem of Fernando \cite{Fe} and Benkart-Britten-Lemire \cite{BBL} that infinite-dimensional simple bounded weight modules exist only for $\mathfrak{g}_n=\mathfrak{sl}({n+1}),\,\mathfrak{sp}({2n})$. In these cases there also exist simple cuspidal modules. For $\mathfrak{g}_n=\mathfrak{o}({2n+1}),\,\mathfrak{o}({2n})$, the category $\mathcal{B}(\mathfrak{g}_n)$ coincides with the category of finite-dimensional $\mathfrak{g}_n$-modules.

By $(\, ,)$ we denote the restriction of the Killing form on $\gh_n$. The induced form on ${\mathfrak h}_n^*$ will be denoted by $(\, ,)$ as well.  By $W_n$ we denote the Weyl group of ${\mathfrak g}_n$.  We only consider Borel subalgebras $\gb_n \subset \gg_n$ such that $\gb_n \supset \gh_n$. Fixing $\gb_n$ is equivalent to fixing positive roots $\Delta_n^+$. 

Let  $Z_n$ be the center of $U_n$.  By $\chi_{\lambda+\rho}: Z_n \to \C$ we
denote the central character of the irreducible $\gb_n$-highest weight
$\gg_n$-module with highest weight $\lambda$, where $\rho$
is the half-sum of positive roots. Recall that
$\chi_{\mu}=\chi_{\nu}$ if and only if $\mu=w(\nu)$ for
some element $w$ of the Weyl group $W_n$.  As usual, we write $w\cdot\lambda$ for the weight $w(\lambda+\rho)-\rho$ for $w\in W_n$, $\lambda\in\fh^*_n$. Finally, recall that a weight $\lambda\in \fh_n^*$ is {\it $\gb_n$-dominant integral} if
$(\lambda,\alpha)\in \mathbb Z_{\geq 0}$ for all $\alpha$ in $\Delta_n^+$.

\subsection{Bounded highest weights modules of $\mathfrak{sl}({n+1})$}
Throughout the subsection, $\gg_n = \mathfrak{sl}({n+1})$. 
In what follows, we fix $\fb_n$ to be the Borel subalgebra of $\fg_n$ with simple roots $\varepsilon_i-\varepsilon_{i+1}$ for  $1\leq i\leq n$. With a slight abuse of notation, we will denote the corresponding Borel subalgebra of $\mathfrak{gl}({n+1})$ also by $\gb_n$. By $L(\lambda)$ we denote the simple highest weight module with highest weight $\lambda$ relative to $\gb_n$. For two weights $\lambda, \mu$ we write $\lambda>\mu$ if $\lambda-\mu$ is a sum of $\gb_n$-positive roots.
 The reflection corresponding to a root $\alpha$ will be denoted by $s_{\alpha}$.  Set $s_i:=s_{\varepsilon_i-\varepsilon_{i+1}}$ and

$$s_i...s_k:=\left\{\begin{array}{ll}
s_is_{i+1}\ldots s_k &\ \text{ for } n\geq k\geq i\geq 1,\\
s_is_{i-1}\ldots s_k & \ \text{ for } 1<k<i<n.
\end{array}
\right.$$

 For a proof of the following proposition we refer the reader to \S 3.3 in \cite{GG}.

\begin{prop} \label{cor-int-bnd}
Assume \emph{dim}$L(\lambda)<\infty$. The modules of the form $L((s_i\dots s_j)\cdot \lambda)$ for $1\leq i,j\leq n$ are, up to isomorphism, all infinite-dimensional simple bounded weight modules which are $\gb_n$-highest weight modules and have central character $\chi_{\lambda + \rho}$. Furthermore,
 $$\Delta^{L((s_i...s_j)  \cdot \lambda) {\rm - inf}} = \Delta^{L((s_k...s_\ell) \cdot \lambda) {\rm - inf}}$$
 if and only if $i = k$.
\end{prop}

\begin{lem} \label{lem-step3} With notation as above, $\Supp L((s_i...s_j)\cdot \mu) \subset \Supp L(s_i\cdot \mu) $. 
\end{lem}
\begin{proof}
 Let $\gs \gl (i,n-i)$ be the subalgebra of $\gg$ which contains $\fh$ and has roots
$$\left\{\pm (\varepsilon_k - \varepsilon_{\ell}) \; | \; k < \ell  \leq i  \right\} \sqcup \left\{\pm (\varepsilon_r - \varepsilon_{s}) \; | \; r > s > i \right\}.$$

By a result of Fernando, see Theorem 4.18 in \cite{Fe}, we have
$$\Supp L((s_i...s_j)\cdot \mu)) = \Supp L_{\gs \gl (i,n-i)} ((s_i...s_j)\cdot \mu) + C(L((s_i...s_j)\cdot \mu))).$$
The inclusion $\Supp L((s_i...s_j)\cdot \mu) \subset \Supp L(s_i\cdot \mu) $ follows from the facts that  $(s_i...s_j)\cdot \mu < s_i\cdot \mu$ and $C(L((s_i...s_j )\cdot \mu)) = C(L(s_i \cdot \mu))$ (by Proposition \ref{cor-int-bnd}), and that the supports of $L((s_i...s_j)\cdot \mu)$ and $L(s_i\cdot \mu))$ are invariant with respect to the Weyl group of $\gs \gl (i,n-i)$. 
\end{proof}

\subsection{Localization of weight modules}
In this subsection $\gg_n = \mathfrak{sl}({n+1})$ or  $\gg_n = \mathfrak{sp}({2n})$.
Let  $F = \{ f_1,...,f_k\}$ be a  subset of pairwise commuting elements of $U_n$ with the condition that each $f_i$ is a locally nilpotent endomorphism of $U_n$.   Let 
$F_{U_n}$ be the multiplicative subset  of $U_n$ generated by $F$, i.e.  $F_{U_n}$  consists of the monomials $f_1^{n_1}...f_k^{n_k}$, $n_i \in {\mathbb Z}_{\geq 0}$. By $D_{F} U_n$ we denote the localization of $U_n$ relative to $F_{U_n}$. Note that $F_{U_n}$ satisfies Ore's localizability conditions  as the operators $\textup{ad} f_i$ are locally nilpotent, see for example Lemma 4.2 in \cite{M}. 

For a $U_n$-module $M$, by $D_F M: = D_F U_n \otimes_{U_n} M$ we denote the localization of $M$ relative to $F_{U_n}$. We will consider $D_F M$ both as a $U_n$-module and as a $D_F U_n$-module. 
By $\theta_{F} : M \to {D}_{F} M$ we denote the map defined by $\theta_{F} (m) = 1 \otimes m $. Note that $\theta_{F}$ is an injection if and only if every element of $F$ acts injectively on $M$. In the latter case, $M$ will be considered naturally as a submodule of ${D}_{F} M$.  

It is well known that ${D}_{F}$ is an exact functor from the category of  $U_n$-modules to the category of ${D}_{F} {U_n}$-modules. The following lemma follows from Lemma 4.4(ii) in \cite{M}.

\begin{lem} \label{lem-loc-sl2} Let $\alpha \in \gg_n^{\alpha}$ and $F = \{ f_{\alpha} \}$ for some $f_{\alpha} \in \gg_n^{-\alpha}$. Let also $M$ be a bounded $\gg_n$-module such that $-\alpha \in \Delta_n^{M{\rm -inf}}$. Then $\deg M = \deg D_F M$.
\end{lem}

We now introduce the  ``generalized conjugation'' in $D_F U_n$ following \S 4 of \cite{M}. For ${\bf x}  \in {\mathbb C}^{k}$ we define the automorphism $\Theta_F^{\bf x}$ of 
${D}_{F}  U_n $ in the following way.  For $u
\in {D}_{F}U$, $f \in F$, and $x \in {\mathbb C}$, we set
\begin{eqnarray*}
\Theta_{\{f\}}^{x}(u):= \sum\limits_{
i \geq 0} \binom{x}{i} \,(\mbox{ad}\,f)^{i} (u) \,f^{-i},
\end{eqnarray*}
 where $\binom{x}{i} :=
x(x-1)...(x-i+1)/i!$ for $x \in {\mathbb C}$ and $i \in {\mathbb Z}_{\geq 0}$. Note that the sum on the right-hand side is well defined since $f$ is ad-nilpotent on $U_n$. Now, for  ${\bf x}  = (x_1,...,x_k) \in {\mathbb C}^k$ and $F = \{ f_1,...,f_k\}$, let

\begin{eqnarray*}
\Theta_{F}^{{\bf x}}(u): =\prod_{1\leq j \leq k} \Theta_{\{f_j\}}^{x_j}(u) 
\end{eqnarray*}
(the product being well-defined because $F$ is a commutative set). Note that, if ${\bf f}^{\bf x}:=f_{1}^{x_1}...f_{k}^{x_k}$ for ${\bf x} \in {\mathbb Z}^k$ then  $\Theta_{F}^{\bf x}(u) = {\bf f}^{\mathbf x}u {\bf f}^{-\mathbf{x}}$.

For a
${D}_{F} U_n$-module $N$ we denote by $\Phi^{\bf x}_{F} N$ 
the ${D}_{F} {U_n}$-module $N$ twisted by $\Theta_{F}^{\bf x}$.  The action on  $\Phi^{\bf x}_{F} N$ is given by
 $$
u \cdot v^{\bf x} :=
 ( \Theta_{F}^{\bf x}(u)\cdot v)^{\bf x},
$$
where $u \in {D}_{F} {U_n}$, $v \in N$, and $w^{\bf x}$
stands for the element $w$ of $N$ considered as an element of
$\Phi^{\bf x}_{F} N$.
In the case $J = \{1,2,...,k \}$ and ${\bf x} \in {\mathbb Z}^k$, there is a natural isomorphism of ${D}_{F} {U_n}$-modules $M \to \Phi_{F}^{\bf x} M$ given by $m \mapsto ({\bf f}^{\bf x} \cdot m)^{\bf x}$, with inverse map defined by $n^{\bf x} \mapsto {\bf f}^{-\bf{x}} \cdot n$.
In view of this isomorphism, for ${\bf x} \in {\mathbb Z}^k$, we will identify $M$ with  $\Phi_{F}^{\bf x}M$, and for any ${\bf x} \in {\mathbb C}^k$ will write ${\bf f}^{\bf x} \cdot m$ (or simply ${\bf f}^{\bf x} m$) for $m^{-\bf{x}}$ whenever $m \in M$.  Properties of the twisting functor $\Phi_{F}^{\bf x}$ are listed below. The proofs of (i) and (ii) can be found in \S 4  of \cite{M}, while (iii) is a standard fact.

\begin{lem} \label{lemma-conj} Let $F = \{ f_1,...,f_k\}$ be a set of locally ad-nilpotent commuting elements of $U_n$, $M$   be  a ${D}_{F} {U_n}$-module, $m \in M$, $u \in  U_n$, and  ${\bf x}, {\bf y} \in {\mathbb C}^k$. 

{\rm(i)}  $\Phi^{\bf x}_{F} \Phi^{\bf y}_{F}M   \simeq \Phi^{\bf{x+y}}_{F} M$, in particular, $\Phi^{\bf x}_{F}
\Phi^{-\bf{x}}_{F} M \simeq M$;

{\rm(ii)}   ${\bf f}^{\bf x} \cdot ( u \cdot ({\bf f}^{-{\bf x}} \cdot m) ) = \Theta_{F}^{\bf x}(u) \cdot m$;

{\rm(iii)}  $\Phi^{\bf x}_{F}$ is an exact functor;

 \end{lem}

For any $U_n$-module $M$, and ${\bf x} \in {\mathbb C}^k $ we define
the {\it twisted localization ${D}_{F}^{\bf x} M$ of $M$
relative to $F$ and $\bf x$} by ${D}_{F}^{\bf x} M:=
\Phi^{\bf x}_{F} {D}_{F} M $. The twisted localization is an exact functor from $U_n$-mod to $D_F U_n$-mod. In the case when $F = \{f_{\alpha_1},...,f_{\alpha_k} \}$ for some $\alpha_i$ such that $\alpha_i + \alpha_j \notin \Delta$ and  $f_{\alpha_i} \in \gg_n^{-\alpha_i}$, we will write ${D}_{\Sigma}M$ and ${D}_{\Sigma}^{\bf x} M$ for ${D}_{F} M$ and ${D}_{F}^{\bf x} M$, respectively, where $\Sigma = \{\alpha_1,...,\alpha_k \}$. By definition, ${D}_{\Sigma} M = M$ if $\Sigma = \emptyset$.

 The following provides a classification of  simple bounded $\gg_n$-modules in terms of twisted localization of highest weight modules, see  \cite[Theorem 13.3]{M}. 

\begin{thm}[Mathieu, \cite{M}] \label{thm-clas-bounded}
Let $M$ be a simple infinite-dimensional bounded $\gg_n$-module. Then there is a simple highest weight module $L$, a subset $\Sigma = \{ \alpha_1,...,\alpha_k\}$ of $\Delta_n^{L-{\rm inf}}$, and ${\bf x} \in {\mathbb C}^{|\Sigma|}$ such that $M \simeq {D}_{\Sigma}^{\bf x} L$. Moreover, the central characters of $M$ and $L$ coincide, $\deg M = \deg L$, and 
$$
\Supp M =  x_1\alpha_1+\cdots + x_k \alpha_k +  \Supp L + {\mathbb Z} \Sigma.
$$
\end{thm}

\begin{cor} \label{cor-cones}
Let $\gg_n = \mathfrak{sl}({n+1})$ or $\gg_n = \mathfrak{sp}({2n})$ and let $M_1$ and $M_2$ be two infinite-dimensional bounded $\gg_n$-modules with the same central character and such that $\Supp M_1$ and $\Supp M_2$ are in a single  coset $\lambda + Q_n$ of $Q_n$ in $\fh_n^*$. Then either $C(M_1) = C(M_2)$, or $C(M_1) \cap C(M_2)$ is contained in a hyperplane of $\fh_n^*$. 
\end{cor}
\begin{proof}  The statement follows from description of the singular parts of the semisimple irreducible coherent families in Sections 9 and 10 in \cite{M}, but for the reader's convenience, a short proof is provided.

 We  apply Theorem \ref{thm-clas-bounded} simultaneously to $M_1$ and $M_2$ and find highest weight modules $L_i$, $\Sigma \subset \Delta_n^{L_i-{\rm inf}}$, and  ${\bf x} \in {\mathbb C}^{|\Sigma|}$, so that $M_i \simeq D_{\Sigma}^{\bf x}L_i$, $i = 1,2$. The fact that we can choose the same $\Sigma$ and ${\bf x}$ for both $M_1$ and $M_2$ is not trivial and follows from the notion of coherent family, see Section 4 in \cite{M}. More precisely, since $M_1$ and $M_2$ have the same central character and $\Supp M_1 - \Supp M_2 \subset Q_n$, we may assume that $M_1$ and $M_2$ are in the same irreducible semisimple coherent family (replacing $M_2$ with a module $M_2'$ such that $C(M_2) = C(M_2')$ if necessary).  From $M_i \simeq D_{\Sigma}^{\bf x}L_i$ we  easily  check that $C(M_i)$ is generated by $(-\Sigma) \sqcup \Delta_n^{L_i-{\rm inf}}$, $i=1,2$. Hence, it is enough to prove the corollary for highest weight modules. 
 
 In this case, the statement follows from the following description of the sets $\Delta_n^{L-{\rm inf}}$ of any infinite-dimensional  simple highest weight bounded $\gg_n$-module $L$.  
 
 Let  $I_{\mathfrak{sl}}:=\{1,...,n+1\}$, $I_{\mathfrak{sp}}:=\{1,...,n\}$. For subsets $I$ of $I_{\mathfrak{sl}}$  and $J$ of  $I_{\mathfrak{sp}}$, set 
$$
\Delta_{\mathfrak{sl}}(I) := \left\{\varepsilon_i - \varepsilon_j \; | \; i \in I, j \notin I  \right\}, \; \Delta_{\mathfrak{sp}}(J) :=  \left\{\varepsilon_{i} + \varepsilon_{j}, \varepsilon_{i} - \varepsilon_{k}, - \varepsilon_k - \varepsilon_{\ell} \; | \; i,j \in J; \, k,\ell \notin J  \right\}.
$$
Note that by definition $\Delta_{\mathfrak{sl}}(\emptyset) = \Delta_{\mathfrak{sl}}(I_{\mathfrak{sl}}) = \emptyset$. Then the following holds:
\begin{itemize}
\item[(i)] If $\gg_n = \mathfrak{sl}({n+1})$ and $L$ has integral central character, then $\Delta_n^{L-{\rm inf} }= \Delta(I)$ for some nonempty proper subset $I$ of $I_{\mathfrak{sl}}$.
\item[(ii)] If $\gg_n = \mathfrak{sl}({n+1})$ and $L$ has nonintegral central character, then $\Delta_n^{L-{\rm inf} }= \Delta(I) \cup \Delta(I \cup \{i_0\})$ for some proper subset 
 $I$ of $I_{\mathfrak{sl}}$ and some $i_0 \in I_{\mathfrak{sl}}$.
\item[(iii)]  If $\gg_n = \mathfrak{sp}({2n})$, then $\Delta_n^{L-{\rm inf} }= \Delta_{\mathfrak{sp}}(J)$ for some subset $J$ of $I_{\mathfrak{sp}}$.
\end{itemize}

Claims (i)--(iii) are direct corollaries from Proposition 8.5 and Lemma 9.2 in \cite{M}. \end{proof}

\begin{cor} \label{cor-ext-fd}
Let $\gg_n = \mathfrak{sl} (n+1)$, and let $M$ be a simple bounded $\gg_n$-module such that ${\rm Ext}^1_{\mathcal B (\gg_n)} (F,M) \neq 0$ for some simple finite-dimensional $\fg_n$-module $F$. Then $M$ is a highest weight module.
\end{cor}
\begin{proof}
This is a known fact, nevertheless  we provide a proof.  Assume that $M$ is not a highest weight module. Then by Theorem \ref{thm-clas-bounded}, there is a nonempty set $\Sigma$ and ${\bf x} \in \C^{|\Sigma|}$such that $M \simeq {D}_{\Sigma}^{\bf x} L$ for some highest weight module $L$. The condition ${\rm Ext}^1_{\mathcal B (\gg_n)} (F,M) \neq 0$ implies that $M$ and $F$ have the same central character $\chi$. Since twisted localization preserves central characters, $L$ has central character $\chi$ as well. Next ${\rm Ext}^1_{\mathcal B (\gg_n)} (F,M) \neq 0$ forces  $\Supp M$, $\Supp F$ and $\Supp L$ to be in the same $Q_n$-coset in $\gh_n^*$. The isomorphism  $M \simeq {D}_{\Sigma}^{\bf x} L$ implies that $C(L) \subset C(M)$. Then $C(M) = C(L)$ by Corollary \ref{cor-cones}, and hence $M$ is a highest weight module. \end{proof}

An important property of the twisted localization is that it preserves annihilators in $U_n$. We set $\Ann (\cdot):= \Ann_{U_n}(\cdot)$.
\begin{lem} \label{lem-loc-ann}
Let $M$ be an $U_n$-module. Then $\Ann M \subset \Ann {D}_{F}^{\bf x}  M$.
\end{lem}
\begin{proof}
Let $u \in \Ann M$. Observe that $(\ad f)^i(u) \in \Ann M$ for all $f\in F$. This implies that  $u$ belongs to the annihilator of the $U_n$-module $ {D}_{F} M$. Now 
the statement follows from the fact that $u \cdot ({\bf f}^{\bf x}m) = {\bf f}^{\bf x} (\Theta_{F}^{\bf -x}(u) \cdot m) = 0$ for $m \in M$ and ${\bf x} \in {\mathbb C}^k$,  by Lemma \ref{lemma-conj}(ii).
\end{proof}

\begin{cor}
Let $M$ be a simple infinite-dimensional bounded $\gg_n$-module. Then there exists a simple bounded highest weight module $L$ such that $\Ann M = \Ann L$.
\end{cor}
\begin{proof}
From Theorem \ref{thm-clas-bounded},  there is an isomorphism $M \simeq  {D}_{F}^{\bf x}  L$ for some highest weight module $L$. Then  $\Ann L \subset \Ann M$ by Lemma \ref{lem-loc-ann}. On the other hand, by Lemma \ref{lemma-conj}(i), $L$ is a submodule of ${D}_{F} L \simeq   {D}_{F}^{\bf -x} M$. Hence, again by  Lemma \ref{lem-loc-ann} we have  $\Ann M \subset \Ann L$.
\end{proof}

\subsection{Uniform bounded weight modules of $\mathfrak{sl}({n+1})$ and $\mathfrak{sp}({2n})$} \label{subsec-properties} Recall that every bounded  $\gg_n$-module, whose support lies in a single $Q_n$-coset,  has finite length, see   Lemma 3.3 in \cite{M}.

\begin{lem} \label{lem-deg}
Let  $M$ be an infinite-dimensional uniform bounded $\gg_n$-module whose support lies in a single $Q_n$-coset, and let $M_1,...,M_k$, $F_1,...,F_{\ell}$ be the simple constituents of $M$, counted with multiplicities. Assume that the modules $M_i$ are infinite dimensional and the modules $F_j$ are finite dimensional. Then $\Delta_n^{M-{\rm inf}}$ = $\Delta_n^{M_i-{\rm inf}}$ for all $i$, and 
$$
\deg M = \deg M_1 + \cdots + \deg M_k .
$$
\end{lem}
\begin{proof} The uniformity of $M$ implies that $M$ has no finite-dimensional submodules. Since $M$ has finite length, without loss of generality we may assume that $M_1$ is a simple submodule of $M$. It is easy to see that $\Delta_n^{M{\rm-inf}} =\Delta_n^{M_1{\rm-inf}}$, $\Delta_n^{M{\rm-fin}} =\Delta_n^{M_1{\rm-fin}}$, $\Delta_n^{M_i{\rm-inf}} \subset \Delta_n^{M{\rm-inf}}$ and $\Delta_n^{M{\rm-fin}} \subset \Delta_n^{M_i{\rm-fin}}$ for $i>1$. Assume that $\Delta_n^{M{\rm-inf}} \neq\Delta_n^{M_i{\rm-inf}}$ for some $i>1$. Then $C(M_1) \neq C(M_i)$ and, by Corollary \ref{cor-cones}, $C(M_1) \cap C(M_i)$ is contained in a hyperplane. Therefore there is a root in $\Delta_n^{M_i{\rm-inf}} $ which does not lie in $\Delta_n^{M_1{\rm-inf}} = \Delta_n^{M{\rm-inf}}$, contradicting the inclusion  $\Delta_n^{M_i{\rm-inf}} \subset \Delta_n^{M{\rm-inf}}$ . Thus $\Delta_n^{M{\rm-inf}} = \Delta_n^{M_i{\rm-inf}}$ for all $i$, and consequently, $C(M_i)=C(M)$ for all $i$.

Let $\Sigma$ be a basis of $Q_n$ such that it consists of commuting roots of  $\Delta_n^{M{\rm-inf}} =  \Delta_n^{M_i{\rm-inf}}$. Such a basis exists by Lemma 4.4(i) in \cite{M}. We apply the localization functor ${D}_{\Sigma}$ with respect to the Ore subset $F$ of $U_n$ generated by $f_{\alpha} \in \gg_n^{-\alpha}$, $\alpha \in \Sigma$.  The functor ${D}_{\Sigma}$ is  exact  and  $D_{\Sigma} F_j = 0$, $j=1,...,\ell$. Therefore $D_{\Sigma}M$ has a filtration with consecutive quotients $D_{\Sigma}M_i$ for $i=1,...,k$. Each of the $k+1$ modules $D_{\Sigma}M$, $D_{\Sigma}M_1,\dots, D_{\Sigma}M_k$ has support $\lambda + Q_n$ for any $\lambda \in \Supp M$,  and its weight spaces have equal dimension. Hence,  $\deg D_{\Sigma} M = \sum_{i=1}^k \deg D_{\Sigma} M_i$. On the other hand, after multiple applications of Lemma \ref{lem-loc-sl2}, we obtain $\deg D_{\Sigma}M = \deg M $ and $\deg D_{\Sigma}M_i = \deg M_i$ for $i=1,...,k$. Therefore, 
$$
\deg M = \deg D_{\Sigma} M = \sum_{i=1}^k \deg D_{\Sigma} M_i =   \sum_{i=1}^k \deg  M_i.
$$
\end{proof}

Let  $U^0_n$ denote the centralizer of $\mathfrak h_{n}$ in $U_n$. 

The following proposition is crucial for proving Corollary \ref{cor-loc-simple} below.
\begin{prop} \label{prop_ext_simple}
Let $\gg_n = \mathfrak{sl}({n+1})$ and let $M'$ be an infinite-dimensional  bounded uniform weight $\gg_n$-module. Assume $(M')^{\lambda}$ is a simple $U^0_n$-module for some $\lambda\in \textup{Supp}_\textup{ess}M'$. Then, for every $m \in (M')^{\lambda}$, the module $U_n\cdot m$ is either  simple, or is an extension of a simple finite-dimensional module by a simple infinite-dimensional  module.
\end{prop}
\begin{proof}
 Note that since $M'$ is bounded uniform  module, the module $M:=U_n\cdot m$ is a bounded uniform module whose support lies in a single $Q_n$-coset. Therefore, $M$ has finite length.  Consider a short exact sequence
$$
0 \to M_1 \to M \to M_2 \to 0,
$$
where $M_1$ is a simple module. Since $M$ is infinite dimensional, $M_1$ is also infinite dimensional, as otherwise $M$ would not be uniform. We have a short exact sequence of $U^0_n$-modules
$$
0 \to M_1^{\lambda} \to M^{\lambda}=(M')^{\lambda}  \to M_2^{\lambda}  \to 0.
$$
As $(M')^{\lambda}$ is a simple $U^0_n$-module, either $M_1^{\lambda} =0$ or $M_2^{\lambda} =0$. In the latter case $M_1^{\lambda} = M^{\lambda} $, hence $M = M_1$ is simple.

 Suppose  $M_1^{\lambda} =0$.  In the rest of the proof we show in several steps that $M_2$ is simple finite dimensional.

{\it Step 1:} Let $M_2$ be uniform. Assume to the contrary that dim$M_2=\infty$. Then Lemma \ref{lem-deg}, applied to both $M$ and $M_2$, implies $\deg M = \deg M_1 + \deg M_2$. On the other hand, $ \deg M = \dim M^{\lambda} = \dim M_2^{\lambda} \leq \deg M_2$. Therefore, $\deg  M = \deg M_2$ and $\deg M_1 = 0$, which is a contradiction, hence dim$M_2<\infty$. The simplicity of $M_2$ follows form the simplicity of $M_2^\lambda$ as a $U^0_n$-module.

{\it Step 2:} Assume  that $M_2$ is not uniform. Then $M_2$ must have a simple finite-dimensional subquotient, as otherwise Lemma \ref{lem-deg}, applied to $M$, would imply that $M_2$ has infinite-dimensional composition factors $M^k$ with $\Delta^{M^{k}-\textup{inf}}=\Delta^{M^{\ell}-\textup{inf}}$ for all $k,\ell$. In turn, this would imply that $M_2$ is uniform.

Moreover, $M_2$ must have a simple finite-dimensional submodule $F$. Otherwise, $M_2$ would again be uniform. The uniformity of $M$ implies ${\rm Ext}^1_{\mathcal B (\gg_n)} (F, M_1) \neq 0$. By Corollary \ref{cor-ext-fd}, $M_1$ is a simple highest weight module. As the category $\mathcal B (\gg_n)$ is stable under the automorphisms of $\fg$ from the Weyl group $W_n$, after twisting $M$ by an appropriate automorphism, we may assume that $M_1$ is highest weight module relative to the  Borel subalgebra $\gb_n$.  Since $\dim M_1=\infty$,  Proposition \ref{cor-int-bnd} shows that $M_1 \simeq L((s_i...s_j) \cdot \mu)$ for some $i,j$, $1 \leq i,j \leq n$, and some dominant integral $\mu$. Then $F \simeq L(\mu)$. 

{\it Step 3:} We note that $M_1\simeq L(s_i \cdot \mu)$ for some $i$. Consider first the case $\mu=0$.  By  Lemma \ref{lem-step3}, there is no $\alpha\in\Delta$ such that $\alpha\in \Supp L((s_i...s_j)\cdot 0)$ for $i\neq j$, which shows that Ext$^1_{\mathcal{O}}(L(0),L((s_i...s_j)\cdot 0)) = 0$ for $i\neq j$. For a general $\mu$, the statement follows by applying an appropriate translation functor. In the rest of the proof we fix $i$.

{\it Step 4:} We note that  $\lambda \in \Supp L(\mu)$ but $\lambda \notin \Supp L((s_i...s_j) \cdot \mu)$ for all $j$. Indeed, by assumption $\lambda\notin \Supp M_1$, where we write $\varkappa<\eta$ if $\eta-\varkappa$ is a sum of positive roots, therefore Lemma $2.3$ implies  $\lambda\notin \Supp ((s_i\dots s_j)\cdot \mu)$ for all $j$.

{\it Step 5:} Next, we claim that $[M:L (\mu))] = [M: L(s_i \cdot \mu)]$ and $[M:L(\nu)] = 0$ for all $\nu \neq \mu, s_i\cdot \mu$. Let $\alpha_i = \varepsilon_i - \varepsilon_{i+1}$ and let $0\neq x \in \gg^{\alpha_i}$. Then $x : M^{s_i (\mu)} \to M^{s_i \cdot \mu}$ is injective. On the other hand, since $s_i (\mu) > s_i \cdot \mu  > (s_i...s_j)\cdot \mu$, the following holds:
$$
[M: L(s_i \cdot \mu)] = \dim M^{s_i \cdot \mu}, \; [M: L(\mu)] = \dim M^{s_i (\mu)}.
$$ 
 The injectivity of $x : M^{s_i (\mu)} \to M^{s_i \cdot \mu}$ yields
\begin{equation}\label{eq-claim}
[M: L(\mu)]  \leq [M: L(s_i \cdot \mu)] .
\end{equation}
Set $d := \deg M = \dim M^{\lambda}$. The previous step shows that  $\deg M = d  [M: L(\mu)]$. On the other hand, one can check that $L(\mu)$ is isomorphic to a subquotient of the module $D_{f_{\varepsilon_{i} - \varepsilon_{i+1}}} \ldots D_{f_{\varepsilon_{j} - \varepsilon_{j+1}}} L((s_i...s_j) \cdot \mu) $, and by Lemma  \ref{lem-loc-sl2} we have $\deg L( (s_i...s_j) \cdot \mu)) \geq \deg L(\mu).$ Thus
$$
d [M: L(\mu)] =\deg M \geq  \sum_j \deg L(  (s_i...s_j) \cdot \mu) [M: L( (s_i...s_j) \cdot \mu)] \geq d \sum_j [M: L( (s_i...s_j) \cdot \mu)].
$$
This together with (\ref{eq-claim}) implies the statement of this step.

{\it Step 6: } Consider now the socle filtration of $M'$. Since there are no self-extensions of $L(s_i\cdot \mu)$ or of $L(\mu)$ in the category of weight modules, all nonzero odd layers of the socle filtration of $M'$ are direct sums of copies of $L(s_i\cdot \mu)$, and all even nonzero layers are sums of copies of $L(\mu)$. This shows that the Loewy length of $M$ is at most 3. Indeed, otherwise the submodule of $M'$ generated by the preimage in $M'$ of Layer 4 would be a quotient of a direct sum of Verma modules with highest weight $\mu$, and those do not have finite-dimensional subquotients of their radicals. Next, the irreducibly of $M^\lambda$ as $U^0_n$-module shows that there is a single copy of $L(\mu)$ in Layer 2. This together with Step 4 implies that $M_2$ is isomorphic to $L(\mu)$, and we obtain a contradiction with our assumption that $M_2$ is not uniform.

The result follows.
\end{proof}

\begin{cor} \label{cor-sp-deg1}
Let $\gg_n = \mathfrak{sp}({2n})$ and $m$, $\lambda$, and $M'$ be as in Proposition \ref{prop_ext_simple}. Then $U_n\cdot m$ is  a simple $\gg_n$-module.
\end{cor}
\begin{proof}
The corollary follows from Lemma \ref{lem-deg} and from the fact that there are no nontrivial extensions between a simple infinite-dimensional bounded $\fg_n$-module and a finite-dimensional $\fg_n$-module. The latter is a consequence of the observation that the  central characters of simple bounded infinite-dimensional $\fg_n$-modules are never integral \cite{M}.
\end{proof}

\subsection{The Weyl Algebra $\mathcal D_{n+1}$}

Denote by $\mathcal D_{n+1}$ the algebra of polynomial differential operators of ${\mathbb C}[x_1,...,x_{n+1}]$. Namely, $\mathcal D_{n+1} = {\mathbb C}[x_1,...,x_{n+1}, \partial_1,...,\partial_{n+1}]$ with relations $x_i\partial_j - x_j\partial_i = \delta_{ij}$, $x_i x_j=x_j x_i$, $\partial_i\partial_j=\partial_j\partial_i$.

In what follows, we set
$x^{\lambda} := x_1^{\lambda_1}...x_{{n+1}}^{\lambda_{{n+1}}}$ for $\lambda = (\lambda_1,...,\lambda_{{n+1}})\in{\mathbb C}^{n+1}$.
Let $\mu = (\mu_1,...,\mu_{{n+1}})\in\C^{{n+1}}$.   The space of shifted Laurent polynomials
$$
\mathcal{F}(\mu):= \{ x^{\mu} p \; | \; p \in \C [x_1^{\pm1},...,x_{{n+1}}^{\pm 1}]\}.
$$
 is a ${\mathcal D}_{n+1}$-module in a natural way.

\subsection{Simple multiplicity-free modules of $\mathfrak{sl}({n+1})$}\label{subsec-sln-mult-free}  In this subsection $\fg_n=\mathfrak{sl}(n+1)$. The classification of simple multiplicity-free weight $\fg_n$-modules was first obtained in \cite{BBL}. Recall the homomorphism $U_n \to {\mathcal D}_{n+1}$ defined by the correspondence $e_{ij} \to x_i \partial_j$, where $ e_{ij} = e_{\varepsilon_i - \varepsilon_j}$ are the elementary $(n+1)\times(n+1)$-matrices. The $\mathcal{D}_{n+1}$-module  ${\mathcal F} (\mu)$ becomes a $\fg_n$-module through this homomorphism, and in particular,
$$
e_{ij} \cdot x^{\mu} =  \mu_j x^{\mu + \varepsilon_i- \varepsilon_j}.
$$

Moreover,
$$
{\mathcal F}_{\mathfrak{sl}} (\mu):= \{ x^{\mu} p \; | \; p \in \C [x_1^{\pm1},...,x_{n+1}^{\pm 1}], \deg p = 0\} =  {\rm Span} \{ x^{\lambda} \; | \; \lambda - \mu \in Q_{\mathfrak{gl}(n+1)} \}
$$
is a  $\fg_n$-submodule of ${\mathcal F} (\mu)$. It is clear that ${\mathcal F}_{\mathfrak{sl}} (\mu) \simeq {\mathcal F}_{\mathfrak{sl}}(\mu')$ if and only if  $\mu - \mu' \in Q_{\mathfrak{gl}(n+1)} $. 

Recall that ${\rm Int} (\mu)$, ${\rm Int}^+ (\mu)$ and ${\rm Int}^- (\mu)$ stand for the subsets of $\mathbb Z_{> 0}$ consisting of all $i$ such that $\mu_i \in {\mathbb Z}$, $\mu_i \in {\mathbb Z}_{\geq 0}$, $\mu_i \in {\mathbb Z}_{< 0}$,  respectively.  Set
$$
V_{\mathfrak{sl}} (\mu): = {\rm Span} \{ x^{\lambda} \; | \; \lambda - \mu \in Q_{\mathfrak{gl}(n+1)},   {\rm Int}^+ (\mu) \subset  {\rm Int}^+ (\lambda) \}.
$$
Note that $V_{\mathfrak{sl}} (\mu)$ is an  $\mathfrak{sl}({n+1})$-submodule of ${\mathcal F}_{\mathfrak{sl}} (\mu)$. Furthermore, $V_{\mathfrak{sl}} (\mu') \subset V_{\mathfrak{sl}} (\mu)$ if and only if $\mu - \mu' \in Q_{\mathfrak{gl}(n+1)} $ and ${\rm Int}^+ (\mu) \subset  {\rm Int}^+ (\mu')$. 

\begin{defn} \label{def-f} Let  $V_{\mathfrak{sl}}(\mu)^+ = 0$ whenever ${\rm Int}^+ (\mu) = {\rm Int}  (\mu)$. For  $\mu$ with ${\rm Int}^+ (\mu) \subsetneq {\rm Int}  (\mu)$,  let
$$
V_{\mathfrak{sl}}(\mu)^+ := \sum_{V_{\mathfrak{sl}} (\mu') \subsetneq V_{\mathfrak{sl}} (\mu)} V_{\mathfrak{sl}} (\mu'). 
$$
Then set $X_{\mathfrak{sl}}(\mu) := V_{\mathfrak{sl}} (\mu)/ V_{\mathfrak{sl}} (\mu)^+$.
\end{defn}

The $\mathfrak{sl}({n+1})$-module $X_{\mathfrak{sl}}(\mu)$ is clearly also a $\mathfrak{gl}({n+1})$-module, and moreover it is simple (both as a $\mathfrak{gl}(n+1)$- and an $\mathfrak{sl}(n+1)$-module). The modules $X_{\mathfrak{sl}}(\mu)$ have been first studied in \cite{BBL} (where $X_{\mathfrak{sl}}(\mu)$ is denoted by $N(\mu)$). 

In what follows,  for $\mu = (\mu_1,...,\mu_{n+1})$ we sometimes write $X_{\mathfrak{sl}}(\mu_1,...,\mu_{n+1})$ instead of $X_{\mathfrak{sl}}(\mu)$. The same convention applies to other  modules like $L(\mu)$,  etc.

\begin{defn}
Let $\mu, \mu' \in \C^{n+1}$.
\begin{itemize}
 \item[(i)]  We write $\mu  \sim_{\mathcal D} \mu'$ if $\mu_i - \mu'_i \in \Z$ for all $i$, ${\rm Int}^+ (\boldsymbol{\mu}) = {\rm Int}^+ (\boldsymbol{\mu'})$ and ${\rm Int}^- (\boldsymbol{\mu}) = {\rm Int}^- (\boldsymbol{\mu'})$.
 \item[(ii)] We write $\mu \sim_{\mathfrak{sl}} \mu'$ if  $\mu - \mu' \in Q_{\mathfrak{gl}(n+1)}$ and ${\rm Int}^+ (\mu) = {\rm Int}^+ (\mu')$. In particular, $\mu \sim_{\mathfrak{sl}} \mu'$ if and only if $\mu  \sim_{\mathcal D} \mu'$ and $\mu - \mu' \in Q_{\mathfrak{gl}(n+1)}$.
 \end{itemize}
\end{defn}

In what follows we will sometimes consider elements of ${\mathbb C}^{n+1}$ as weights of $\mathfrak{sl} (n+1)$, and we recall 
 (see Subsection \ref{subsec-general-sln}) that this means that we consider the projection of the respective sequence into $\gh_n^*$.

The next theorem follows from Theorem 5.8 and Proposition 3.4 in \cite{BBL}.
\begin{thm} \label{thm-bounded-finite}
Every simple multiplicity-free weight $\mathfrak{sl}({n+1})$-module is isomorphic either to $X_{\mathfrak{sl}}(\mu)$ for some $\mu \in \C^{n+1}$, or to $\Lambda^i(V_{n+1})$ for some $i$, $2\leq i \leq n-1$.
\end{thm}

Some properties of the $\mathfrak{sl}({n+1})$-modules  $X_{\mathfrak{sl}}(\mu)$ are listed in the next proposition.

\begin{prop} \label{prop-x-sl}  Let $\mu \in \C^{n+1}$ and $n>1$. Then the following statements hold: 

\begin{itemize}
\item[(i)] The root space $\mathfrak{sl}(n+1)^{\varepsilon_i-\varepsilon_j}$ acts locally finitely on $X_{\mathfrak{sl}}(\mu)$ if and only if $i \in {\rm Int}^- (\mu)$ or $j \in {\rm Int}^+ (\mu)$.  In particular, $X_{\mathfrak{sl}}(\mu)$ is  finite-dimensional if and only if ${\rm Int}^- (\mu) = \{1,...,n+1 \}$ or ${\rm Int}^+ (\mu) = \{1,...,n+1 \}$, and  $X_{\mathfrak{sl}}(\nu)$ is cuspidal if and only if $ {\rm Int} (\nu) = \emptyset$; in the latter case  $X_{\mathfrak{sl}}(\nu) = {\mathcal F}_{\mathfrak{sl}} (\nu)$.
\item[(ii)] $\Supp X_{\mathfrak{sl}}(\mu) = \{\lambda \in \C^{n+1} \; | \; \lambda \sim_{\mathfrak{sl}} \mu\}$.
\item[(iii)] The central character of $X_{\mathfrak{sl}}(\mu)$ equals $\chi_{|\mu|\varepsilon_1+\rho}$, where $\rho=(\frac{n}{2},\frac{n-2}{2},\dots, -\frac{n}{2})$. 
\end{itemize}
\end{prop}
\begin{proof}Parts (i) and  (ii) follow in a straightforward way from the definition of $X_{\mathfrak{sl}}(\mu)$. For part (iii) one can use the fact that for $a \in \C$, the module ${\mathcal M} (a) := \bigoplus_{|\mu| = a}{\mathcal F}_{\mathfrak{sl}} (\mu)^{\rm ss}$ is a semisimple irreducible coherent family as defined by Mathieu in \cite{M}, Section 4. Here $\cdot^{ss}$ denotes the semisimplification of a module of finite length, and the direct sum runs over a complete set $\mu$ of  representatives of the quotient $\C^{n+1}/Q_{\gg \gl (n+1)}$ satisfying the condition $|\mu|=a$. Since all simple subquotients of such coherent family have the same central character (see Proposition 4.8 in \cite{M}), it is enough to look at the central character of  the simple highest weight submodule $X_{\mathfrak{sl}}(a, 0,...,0)$ of ${\mathcal M} (a) $. The latter central character is  $\chi_{a\varepsilon_1+\rho}$ by definition. \end{proof}

Recall that $x^{(k)}$ stands for the $k$-tuple $(x,x,...,x)$.
\begin{lem} \label{lem-x-sl-iso}  Let $n>1$.
\begin{itemize}
\item[(i)] $X_{\mathfrak{sl}}(\mu)$ and $X_{\mathfrak{sl}}(\mu')$ are isomorphic as $\mathfrak{gl}(n+1)$-modules if and only if $\mu' \sim_\mathfrak{sl} \mu$.
\item[(ii)] $X_{\mathfrak{sl}}(\mu)$ and $X_{\mathfrak{sl}}(\mu')$ are isomorphic as $\mathfrak{sl}({n+1})$-modules  if and only if either $\mu' \sim_{\mathfrak{sl}} \mu$ or $\{\mu,\mu'\} = \{ 0^{(n+1)}, (-1)^{(n+1)}\}$. 
\end{itemize}
\end{lem}
\begin{proof} The ``if'' parts of the lemma easily follow from the definition of $X_{\mathfrak{sl}}(\mu)$. The ``only if'' parts follow by looking at the support and central character of $X_{\mathfrak{sl}}(\mu) \simeq X_{\mathfrak{sl}}(\mu')$, and using Proposition \ref{prop-x-sl}(ii),(iii).
\end{proof}

The previous lemma together with Proposition \ref{prop-x-sl}(ii), (iii) implies the following.
\begin{cor} \label{cor-x-sl}
 If $X_{\mathfrak{sl}}(\mu)$  has central character $\chi_{c\varepsilon_1 + \rho}$ and $\lambda \in \Supp X_{\mathfrak{sl}} (\mu)$, then $\mu \sim_{\mathfrak{sl}} \nu$ where $\nu_i := \lambda_i + \frac{1}{n+1} \left( c - |\lambda| \right)$.
\end{cor}

Let $\gb_n'$ be a  Borel subalgebra  of $\mathfrak{sl} (n+1)$.    For   $i_0 \in \{1,2,...,n+1 \}$ and $a \in \C$, define 
$$\varepsilon_{\gb_n'} (i_0,a)_i :=   
\begin{cases} 
 -1& \mbox{ if } \varepsilon_i - \varepsilon_{i_0} \mbox{ is a } \gb_n'\mbox{-positive root},\\
 a& \mbox{ if } i = i_0,\\
 0& \mbox{ if } \varepsilon_i - \varepsilon_{i_0} \mbox{ is a } \gb_n'\mbox{-negative root}.
\end{cases}
 $$
By definition $\varepsilon_{\gb_n'} (i_0,a) \in \C^{n+1}$, and $\varepsilon_{\gb_n'} (i_0,a) = \varepsilon_{\gb_n'} (j_0,b)$ if and only if $(i_0,a) = (j_0,b)$, or $i_0=j_0+1$ and $a =0, b=-1$. 

\begin{prop} \label{prop-sln-hw} Let $\gb_n'$ be a  Borel subalgebra of $\mathfrak{sl}({n+1})$.  Then $X_{\mathfrak{sl}}(\mu)$ is $\gb_n'$-highest weight module if and only if $\mu \sim_{\mathfrak{sl}} \varepsilon_{\gb_n'} (i_0,a)$ for some $a \in {\mathbb C}$ and some $i_0 \in \{1,...,n+1\}$,  and in this case $\varepsilon_{\gb_n'} (i_0,a)$ is the $\gb_n'$-highest weight of $X_{\mathfrak{sl}}(\mu)$. Equivalently, $X_{\mathfrak{sl}}(\mu)$ is $\gb_n'$-highest weight module if and only if there is $i_0 \in \{1,...,n+1\}$ so that $\mu_i \in \Z_{<0}$ if  $\varepsilon_i - \varepsilon_{i_0}$ is a $\gb_n'$-positive root and  $\mu_j \in \Z_{\geq 0}$ if $\varepsilon_{i_0} - \varepsilon_{j}$ is a $\gb_n'$-positive root.
\end{prop}
\begin{proof}
We use the definition of $X_{\mathfrak{sl}}(\mu)$ and verify that, if $w \in X_{\mathfrak{sl}}(\mu) $ is such that $e_{ij} \cdot w = 0$ whenever $\varepsilon_i - \varepsilon_{j}$ is $\gb_n'$-positive root, then $w$ has weight $\varepsilon_{\gb_n'} (i_k,a)$ for some $k$ and $a$. The statement also follows form Proposition 3.4 in \cite{BBL}.
\end{proof}

\begin{cor} \label{cor-sn-sln}
If $\dim X_{\mathfrak{sl}}(\mu)<\infty$, then $X_{\mathfrak{sl}}(\mu)\simeq S^m(V_{n+1})$ for $m\geq0$ or $X_{\mathfrak{sl}}(\mu)\simeq S^m(V^*_{n+1})$ for $m\geq 0$.
\end{cor}

\begin{proof} The irreducible highest weight module  with $\gb_n'$-highest weight $\varepsilon_{\gb_n'}(i_0,a)$    is finite dimensional if and only if $i_0=1$ and $a=m\in\mathbb{Z}_{\geq0}$, or if $i_0=n+1$ and $a=-m\in\mathbb{Z}_{<0}$. In the former case, $X_{\mathfrak{sl}}(\mu)\simeq S^m(V_{n+1})$, and in the latter case, $X_{\mathfrak{sl}}(\mu)\simeq S^m(V^*_{n+1})$.
\end{proof}

In order to describe the structure of the restriction $X_{\mathfrak{sl}}(\mu)|_{\gg \gl (n)}$ for any $\mu \in {\mathbb C}^{n+1}$, we set
$$
S(\mu) := \left\{ k \in {\mathbb Z} \; | \; \mu + \mu(k) - k\varepsilon_{n+1} \sim_{\mathfrak{sl}} \mu \mbox{ for some } \mu(k) \in \sum_{i=1}^n {\mathbb Z}\varepsilon_i \right\}.
$$
Note that the definition of $\sim_{\mathfrak{sl}} $ implies $|\mu(k)| = k$ in the definition of $S(\mu) $ above.  Also, note that the set $S(\mu)$ has one of the following three forms: $(-\infty, a] \cap {\mathbb Z}$, $[b,c] \cap {\mathbb Z}$, or $[d,\infty) \cap {\mathbb Z}$, for some integers $a$, $b\leq c$, $d$. For $k \in  S(\mu)$ we put 
$$
S(\mu)[k] := \left\{ \mu(k) \in \sum_{i=1}^n {\mathbb Z}\varepsilon_i  \; | \; \mu + \mu(k) - k\varepsilon_{n+1} \sim_{\mathfrak{sl}} \mu \right\}.
$$

The following lemma is straightforward. 

\begin{lem} \label{lem-set-s}Let $k \in {\mathbb Z}$  and $\mu  = (\overline{\mu}, \mu_{n+1}) \in {\mathbb C}^{n+1}$ for some $\overline{\mu} \in {\mathbb C}^{n}$.
\begin{itemize}

\item[(i)] Let $\nu =  \sum_{i=1}^n \nu_i \varepsilon_i$ with $\nu_i \in {\mathbb Z}$. Then $\nu \in S(\mu)[k]$ if and only if:
$$|\nu|= k, \; \overline{\mu} + \nu \sim_{\mathcal D} \overline{\mu}, \; -k \varepsilon_{n+1} \sim_{\mathcal D}  \mu_{n+1} \varepsilon_{n+1}.$$
\item[(ii)]  If $k \in S(\mu)$ and  $\nu, \nu' \in S(\mu)[k]$, then $X_{\mathfrak{sl}}(\overline{\mu} + \nu) \simeq X_{\mathfrak{sl}}(\overline{\mu} + \nu')$ as $\mathfrak{gl}({n})$-modules (hence as $\gs \gl (n)$-modules).
\end{itemize}
\end{lem}

\begin{lem}\label{lem-decomp} Let  $\mu = (\mu_1,...,\mu_{n+1})$, and $\overline{\mu} = (\mu_1,...,\mu_{n})$. Then 
$$
X_{\mathfrak{sl}}(\mu_1,...,\mu_{n+1})|_{\mathfrak{gl} (n)} \simeq \bigoplus_{k \in {S(\mu)}} X_{\mathfrak{sl}}(\overline{\mu} + \mu(k)),
$$
where in the sum above $\mu(k)$ is any element of $S(\mu)[k]$ (cf. Lemma \ref{lem-set-s}(ii)). Moreover, such a  decomposition arises from the eigenspace decomposition of a central element $E_{n}$ of $\mathfrak{gl}({n})$ considered as an endomorphism of $X_{\mathfrak{sl}}(\mu_1,...,\mu_{n+1})$.
\end{lem}
\begin{proof}
The result is straightforward if we  consider the linear operator  $E_n = \sum_{i=1}^nE_{ii} - n E_{nn}$ on $X_{\mathfrak{sl}}(\mu_1,...,\mu_{n+1})$.
Note that   $E_n$ maps to $\sum_{i=1}^n {x_i \partial_{x_i}} - nx_n \partial_{x_n} $ under the homomorphism $U(\mathfrak{sl}({n+1})) \to {\mathcal D}_{n+1}$.
\end{proof}

The above lemma together with Lemma \ref{lem-x-sl-iso} implies the following.
\begin{cor} \label{cor-mult-free-decomp}
Let $n>2$ and $\mu \in {\mathbb C}^{n+1}$. Then the restrictions $X_{\mathfrak{sl}}(\mu_1,...,\mu_{n+1})|_{\mathfrak{gl} (n)}$ and $X_{\mathfrak{sl}}(\mu_1,...,\mu_{n+1})|_{\mathfrak{sl} (n)}$ are semisimple and each irreducible constituent enters with multiplicity one.
\end{cor} 
\begin{examp} The case $n=2$ in the above corollary is special. For example, the statement of Corollary \ref{cor-mult-free-decomp} holds for the ${\mathfrak{gl}(2)}$-module $X_{\mathfrak{sl}} (-1,1,0)|_{\mathfrak{gl}(2)} = X_{\mathfrak{sl}}(-1,1)  \oplus \bigoplus_{i < 0} X_{\mathfrak{sl}}(i,0)$, and is false for the ${\mathfrak{sl}(2)}$-module  
$X_{\mathfrak{sl}} (-1,1,0)|_{\mathfrak{sl}(2)}$ as  the $\mathfrak{sl}(2)$-modules  $X_{\mathfrak{sl}} (-1,1)$ and $X_{\mathfrak{sl}}(-2,0)$ are isomorphic.
\end{examp}

\begin{lem} \label{lemma-deg}Let $n>3$. Assume that $\lambda = (\lambda_1,...,\lambda_{n+1})$ satisfies  $\lambda_i - \lambda_{i+1} \in \Z_{\geq 0}$ for $i \geq 2$ and  that $L (\lambda)$ is infinite dimensional. Then 
$$
\deg L (\lambda) \geq \dim L(\widetilde{\lambda})
$$
where $\widetilde{\lambda} = (\lambda_3,...,\lambda_{n+1})$ and $L(\widetilde{\lambda})$ is  the finite-dimensional $\mathfrak{gl} (n-1)$-module with highest weight $\tilde\lambda$.

\end{lem}
\begin{proof} 
Let  $\gs :=\mathfrak{gl}(1) \oplus  \mathfrak{gl}(n)$ and $\gl :=\mathfrak{gl}(n) \oplus  \mathfrak{gl}(1)$ be the subalgebras of $\mathfrak{gl}(n+1)$ with root systems $\{\varepsilon_i - \varepsilon_j \; | \; 2\leq i \neq j \leq n+1 \}$ and $\{\varepsilon_i - \varepsilon_j \; | \; 1\leq i \neq j \leq n \}$, respectively. Then $\gp:= \gs +\gb_n$ is a parabolic subalgebra of $\mathfrak{gl}({n+1})$. Set $\widetilde{\gp} := \gp \cap {\mathfrak l}$ and  $\widetilde{\gs} := \gs \cap {\mathfrak l}$. There is an isomorphism  $\widetilde{\gs} \simeq \mathfrak{gl}(1) \oplus  \mathfrak{gl}({n-1}) \oplus \mathfrak{gl}(1)$.

 In the rest of the paper,  for a reductive subalgebra $\ga$ of  $\mathfrak{gl}(n+1)$, we  sometimes use the notation $L_{\ga} (\nu)$ for a simple highest weight $\ga$-module of highest weight $\nu$ relative to the intersection of $\gb_n \cap \ga$. 

For a dominant integral $\mathfrak{gl}({n+1})$-weight $\eta = (\eta_1,...,\eta_{n+1})$ and $1\leq i \leq n$, set $\eta [i] := s_1s_2...s_i \cdot \eta$. Recall that $s_j$ is the simple Weyl reflection corresponding to the root $\varepsilon_j - \varepsilon_{j+1}$. In coordinate form we have
$$
\eta [i] = (\eta_{i+1} -i, \eta_1 + 1,..., \eta_{i} + 1, \eta_{i+2},...,\eta_{n+1}).
$$
By Proposition \ref{cor-int-bnd}, there is a unique dominant integral $\mathfrak{gl}({n+1})$-weight $\eta$ and  a unique $i \geq 1$ such that $\lambda = \eta[i]$. Set $\lambda[1] := \eta[i+1]$ if $i<n$.

Next we note that $L_{\mathfrak{gl}({n+1})}(\lambda)$ is isomorphic to the simple quotient of the parabolically induced module 
$$
M_{\mathfrak p} (\lambda) = U(\mathfrak{gl}({n+1})) \otimes_{U({\mathfrak p})} F(\lambda).
$$
where $F(\lambda)$ denotes the simple finite-dimensional $\mathfrak{s}$-module $L_{\mathfrak{gl}({1})}(\lambda_1) \boxtimes L_{\mathfrak{gl}({n})}(\lambda_2,...,\lambda_{n+1})$. Furthermore, $M_{\mathfrak p} (\lambda) = L_{\mathfrak{gl}({n+1})}(\lambda)$ if $i=n$, while for $i<n$ there is an exact sequence
$$
0 \to L_{\mathfrak{gl}({n+1})}(\lambda [1]) \to M_{\mathfrak{p}}(\lambda) \to L_{\mathfrak{gl}({n+1})}(\lambda) \to 0.
$$
This follows for example from Lemma 11.2 in \cite{M}.
Also, for $\widetilde{\lambda} = (\widetilde{\lambda}_1, \widetilde{\lambda}_2,..., \widetilde{\lambda}_n,  \widetilde{\lambda}_{n+1})$ we set 
$$
M_{\widetilde{\mathfrak p}} (\widetilde{\lambda}) = U(\mathfrak{l}) \otimes_{U({\widetilde{\mathfrak p}})} F(\widetilde{\lambda}).
$$
where $F(\widetilde{\lambda})= L_{\mathfrak{gl}({1})}(\widetilde{\lambda}_1) \boxtimes L_{\mathfrak{gl}({n-1})}(\widetilde{\lambda}_2,..., \widetilde{\lambda}_{n})  \boxtimes  L_{\mathfrak{gl}({1})}(\widetilde{\lambda}_{n+1})$ is a simple $\widetilde{\mathfrak s}$-module.

Applying the Gelfand-Tsetlin rule  we  decompose $F(\lambda)$ into a direct sum of simple $\widetilde{\mathfrak s}$-modules:
$$
F (\lambda) = \bigoplus_{t=1}^kF(\widetilde{\mu}^t),
$$
where $F(\widetilde{\mu}^t) = L_{\mathfrak{gl}({1})}(\lambda_1) \boxtimes L_{\mathfrak{gl}({n-1})}(\widetilde{\mu}^t_2,..., \widetilde{\mu}^t_{n})  \boxtimes  L_{\mathfrak{gl}({1})}(\widetilde{\mu}^t_{n+1})$ and the sum runs over all $\widetilde{\mathfrak s}$-weights $\widetilde{\mu}^t$ with the properties $\lambda_j \leq \widetilde{\mu}^t_j \leq \lambda_{j+1}$ and $\sum_{j=2}^{n+1} \lambda_j = \sum_{j=2}^{n+1} \widetilde{\mu}^t_j$. For convenience we assume that $\widetilde{\mu}^1 = (\lambda_1, \lambda_2,...,\lambda_n, \lambda_{n+1})$ and $\widetilde{\mu}^k = (\lambda_1,\lambda_3,...,\lambda_{n+1}, \lambda_{2})$.

It is not difficult to prove that
$$
\ch M_{\mathfrak p} (\lambda) = \sum_{t=1}^k \sum_{j\geq 0} \ch M_{\widetilde{\mathfrak p}} (\widetilde{\mu}^t + (-j,0,..,0,j)).
$$

We claim that the simple $\gl$-module $L_{\gl}  (\widetilde{\mu}^k + (-j,0,..,0,j))$ is isomorphic to a subquotient of $L_{\mathfrak{gl}({n+1})}(\lambda)$ for all $j\geq 1$. Indeed, assuming the contrary, we see that  $L_{\gl}  (\widetilde{\mu}^k + (-j,0,..,0,j))$ occurs as a subquotient of $L_{\mathfrak{gl}({n+1})}(\lambda[1])$ and, in particular, 
$(\lambda_1 - j, \lambda_3,...,\lambda_{n+1}, \lambda_2+j)$ belongs to  $\Supp L_{\mathfrak{gl}({n+1})}(\lambda[1])$. Using Weyl group invariance, we obtain
\begin{equation} \label{eq-1}
(\lambda_1 - j, \lambda_2 + j, \lambda_3,...,\lambda_{n+1}) \in \Supp L_{\mathfrak{gl}({n+1})}(\lambda[1]).
\end{equation}
On the other hand, we easily check that (\ref{eq-1}) is impossible since $\lambda[1] - (\lambda_1 - j, \lambda_2 + j, \lambda_3,...,\lambda_{n+1}) = (j-x)\varepsilon_1 -j\varepsilon_2 + x \varepsilon_{i+2}$, where $x = \eta_{i+1} - \eta_{i+2} + 1$, and the weight of the latter form are not sum of positive roots. 

We finally note that $\widetilde{\mu}^k + (-N,0,..,0,N) = (\lambda_1-N,\lambda_3,...,\lambda_{n+1}, \lambda_2+N)
$ and that there is $N$ such that 
  the $\mathfrak{gl}(n)$-weight $(\lambda_1-N,\lambda_3,...,\lambda_{n+1})$  is of the form $\nu[n-1]$ for some $\mathfrak{gl}(n)$-dominant integral weight $\nu$ (in fact $N>\lambda_1 - \lambda_{n+1} + n-1$). In particular, 
$$M_{\widetilde{\mathfrak p}} (\widetilde{\mu}^k + (-N;0,..,0;N)) = L_{\widetilde{\gl}} ({\mu}^k + (-N;0,..,0;N))
$$
and, hence,
$$
\deg L(\lambda) \geq \deg M_{\widetilde{\mathfrak p}} (\widetilde{\mu}^k + (-N,0,..,0,N)) =  \dim L_{\mathfrak{gl}({n-1})}(\lambda_3,...,\lambda_{n+1}).
$$
\end{proof}

\subsection{Simple multiplicity-free modules of $\mathfrak{sp}({2n})$}

In this subsection $\fg_n=\mathfrak{sp}(2n)$.
We use the homomorphism $U_n \to {\mathcal D}_n$ defined by the correspondence $e_{\varepsilon_i + \varepsilon_j} \mapsto x_ix_j$ if $i \neq j$, $e_{2\varepsilon_i} \mapsto \frac{1}{2}x_i^2$, $e_{-\varepsilon_i - \varepsilon_j} \mapsto -\partial_i\partial_j$ if $i \neq j$, $e_{-2\varepsilon_i} \mapsto -\frac{1}{2}\partial_i^2$, where $e_{\alpha} \in \gg_n^{\alpha}$ are appropriate nonzero vectors. We also fix ${\mathfrak b}_n$ to be the Borel subalgebra with positive roots $\{ \varepsilon_i - \varepsilon_j, -\varepsilon_k - \varepsilon_{\ell}\; | \; i<j, k\leq \ell\}$ and write $L(\lambda)$ for a simple $\gb_n$-highest weight module with highest weight $\lambda$. We call the  $\fg_n$-modules ${\mathbb C}_{ev}[x_1,...,x_n]$ and ${\mathbb C}_{od}[x_1,...,x_n]$, consisting respectively of polynomials of even and odd degree, the \emph{Shale-Weil modules}. Their respective ${\mathfrak b}_n$-highest weights are $ \left(\frac{1}{2}, \frac{1}{2},\dots, \frac{1}{2}\right)$ and $\left(\frac{3}{2}, \frac{1}{2},\dots, \frac{1}{2}\right)$. The two Shale-Weil modules have the same central character which we will denote by $\chi_{\rm sw}$.

For any $\mu\in \mathbb{C}^n$, we consider ${\mathcal F} (\mu)$ as a $\fg_n$-module through the homomorphism  $U_n \to {\mathcal D}_{n}$. Then 
$$
{\mathcal F}_{\mathfrak{sp}} (\mu):= \{ x^{\mu} p \; | \; p \in \C [x_1^{\pm1},...,x_{n}^{\pm 1}], \deg p \mbox{}\in 2\mathbb{Z}\} = {\rm Span}\, \{ x^{\lambda} \; | \; \lambda - \mu \in Q_{\mathfrak{sp} (2n)}\}$$
is a  $\fg_n$-submodule of ${\mathcal F} (\mu)$. It is easy to check that ${\mathcal F} (\mu) = {\mathcal F} (\mu')$ if and only if $\mu - \mu' \in Q_{\fg_n}$. Similarly to Definition \ref{def-f}, we define the $\fg_n$-module $X_{\mathfrak{sp}} (\mu)$. In particular, $X_{\mathfrak{sp}}(0,0,...,0) \simeq L\left(\frac{1}{2}, \frac{1}{2},\dots, \frac{1}{2}\right)$ and $X_{\mathfrak{sp}}(1,0,...,0) \simeq L\left(\frac{3}{2}, \frac{1}{2},\dots, \frac{1}{2}\right)$.

\begin{defn}
If $\mu$, $\mu'\in \C^{n}$, we write $\mu \sim_{\mathfrak{sp}} \mu'$ if 
$\mu - \mu' \in Q_{\fg_n}$ and  ${\rm Int}^+ (\mu) = {\rm Int}^+ (\mu')$. In particular, if  $\mu \sim_{\mathfrak{sp}} \mu'$ then ${\rm Int}^- (\mu) = {\rm Int}^- (\mu')$.
\end{defn}

The following theorem follows from Proposition 3.6 and Theorem 5.21 in \cite{BBL}.
\begin{thm} \label{thm-x-sp-fin}Let $n>3$ and let  be  a simple multiplicity-free  weight $\mathfrak{sp}({2n})$-module. If $M$ is infinite dimensional, $M$  is isomorphic  to $X_{\mathfrak{sp}}(\mu)$ for some $\mu \in \C^{n}$. If $\dim M < \infty$ then  $M \simeq V_{2n}$ or $M \simeq \mathbb C$.
\end{thm}

\begin{prop} \label{thm-bounded-finite-sp}
Let $\mu \in \C^{n}$ for  $n>3$.
\begin{itemize} 
\item[(i)] If $\alpha \in \Delta_{\mathfrak{sp} (2n)}$, then $\mathfrak{sp}(2n)^{\alpha}$ acts locally finitely on $X_{\mathfrak{sp}}(\mu)$ if and only if 
$$\alpha \in \{\pm \varepsilon_i - \epsilon_j, - 2 \varepsilon_j,  \varepsilon_k \pm  \epsilon_\ell, 2 \varepsilon_k  \; | \: j  \in  {\rm Int}^+ (\mu), \, k  \in  {\rm Int}^- (\mu)\}.
$$
In particular, $X_{\mathfrak{sp}}(\mu)$ is always infinite dimensional and is cuspidal if and only if $ {\rm Int} (\mu) = \emptyset$,  in which case $X_{\mathfrak{sp}}(\mu) = {\mathcal F}_{\mathfrak{sl}} (\mu)$.
\item[(ii)] $\Supp X_{\mathfrak{sp}}(\mu) = \{ \lambda + \left(\frac{1}{2}, \frac{1}{2},\dots, \frac{1}{2}\right) \; | \; \lambda\sim_{\mathfrak{sp}} \mu \} $,
\item[(iii)] The central character of $X_{\mathfrak{sp}}(\mu) $ is $\chi_{\rm sw}$.
\end{itemize}
\end{prop}
\begin{proof} The proof is similar to the proof of Proposition \ref{prop-x-sl}. Parts (i) and  (ii) follow from the definition of $X_{\mathfrak{sp}}(\mu)$. For part (iii) we use that the module $\mathcal{SW}  := \bigoplus_{\mu \in \C^n/Q_n}{\mathcal F}_{\mathfrak{sp}} (\mu)^{\rm ss}$ is a semisimple irreducible coherent family all simple subquotient of which have the same central character. Here the sum runs over a complete set of representatives of the quotient $ \C^n/Q_n$. Since the simple highest weight submodule $X_{\mathfrak{sp}}(0, 0,...,0)$ of $\mathcal{SW}$ has central character  $\chi_{\rm sw}$, the statement follows. \end{proof}

\begin{lem}
The $\mathfrak{sp}({2n})$-modules isomorphism $X_{\mathfrak{sp}}(\mu)$ and $ X_{\mathfrak{sp}}(\mu')$ are isomorphic if and only if $\mu \sim_{\mathfrak{sp}} \mu'$.
\end{lem}
\begin{proof} Like in the proof of Lemma \ref{lem-x-sl-iso}, the ``if'' part follows from the definition of $X_{\mathfrak{sp}}(\mu)$, while the ``only if'' part follows from Proposition \ref{thm-bounded-finite-sp}(ii). \end{proof}

To each Borel subalgebra $\gb_n'$  of $\mathfrak{sp} (2n)$ we assign  integral weights $\omega_{\gb_n'}$ and $\delta_{\gb_n'}$ in ${\mathbb Z}^n$  as follows. Let $\sigma = \sigma_{\gb_n'}$ be the map $\{1,...,n\} \to \{+1,-1 \}$ defined by $\sigma (i) = \pm$ whenever $\pm 2 \varepsilon_i$ is a $\gb_n'$-positive root. Let also $j_0$ be the unique element in $ \{1,...,n\} $ such that $\sigma (i) \varepsilon_i \pm \varepsilon_{j_0}$ are $\gb_n'$-positive roots for all $i \neq j_0$. Then we set $(\omega_{\gb_n'})_i = 0$ if $\sigma(i) = -1$, and $(\omega_{\gb_n'})_i  = - 1$ if $\sigma(i) = +1$. Furthermore, $(\delta_{\gb_n'})_{i} = -\delta_{ij_0}$ if $\sigma(j_0) = +1$ and $(\delta_{\gb_n'})_i = \delta_{ij_0}$ if $\sigma(j_0) = -1$. In particular for the Borel subalgebra $\gb_n$ that we fixed in the beginning of this subsection, we have $\omega_{\gb_n} = (0,0,...,0)$ and $\delta_{\gb_n} = (1,0,...,0)$.

\begin{prop} \label{prop-sp2n-hw} Let $\gb_n'$ be a Borel subalgebra of $\mathfrak{sp}({2n})$. Then $X_{\mathfrak{sp}}(\mu)$ is a $\gb_n'$-highest weight module if and only if $\mu \sim_{\mathfrak{sp}} \omega_{\gb_n'}$ or  $\mu \sim_{\mathfrak{sp}} \omega_{\gb_n'} + \delta_{\gb_n'}$ and in this case, $ \omega_{\gb_n'} + \left( \frac{1}{2},..., \frac{1}{2}\right)$ and $ \omega_{\gb_n'} + \delta_{\gb_n'} + \left( \frac{1}{2},..., \frac{1}{2}\right)$ are the $\gb_n'$-highest weight of $X_{\mathfrak{sp}}(\mu)$, respectively. 
\end{prop}
\begin{proof}
The statement  is straightforward. It also follows from Proposition 3.6 in \cite{BBL}. \end{proof}

In what follows, we fix the subalgebra $\widehat{\gg}_{n-1} = \mathfrak{sp}({2n-2}) \oplus \C$ of $\gg_n$ for which the roots of $\mathfrak{sp}({2n-2})$ contain  $\varepsilon_1,\dots,\varepsilon_{n-1}$. Let $\mu  = (\overline{\mu}, \mu_{n}) \in {\mathbb C}^{n}$ for some $\overline{\mu} \in {\mathbb C}^{n-1}$. 
We write $X_{\widehat{\gg}_{n-1}} (\overline{\mu}; \mu_n)$ for the $\widehat{\gg}_{n-1}$-module $X_{\mathfrak{sp}} (\overline{\mu}) \boxtimes \C_{\mu_n}$, where $\C_{\mu_n}$ is the $\C$-module of weight $\mu_n$. The decomposition $X_{\mathfrak{sp}}(\mu)|_{\widehat{\gg}_{n-1}}$ can be described analogously to  Lemma \ref{lem-decomp}. In the case of $\mathfrak{sp} (2n)$, the analog of $S(\mu)$ can be written in more explicit terms. For this we introduce the following notation: for $z \in {\C}$, we set $\mbox{neg} (z) := -1$ if $z \in \Z_{<0}$, and $\mbox{neg} (z) := 1$ otherwise; for $a \in \Z$, we put $p(a) :=0$ if $a$ is even, and $p(a) :=1$ if $a$ is odd.
\begin{lem} \label{prop-restr-sp2n}
Let $n>2$ and $\mu \in \C^{n}$ be such that $\mu  = (\overline{\mu}, \mu_{n})$ for some $\overline{\mu} \in {\mathbb C}^{n-1}$.
Then 
$$
X_{\mathfrak{sp}}(\mu)|_{\widehat{\gg}_{n-1}} \simeq \bigoplus_{\mu_n' \sim_{\mathcal D} \mu_n} X_{\widehat{\gg}_{n-1}}(\overline{\mu} + p(\mu_n - \mu_n'){\rm neg} (\mu_1)\varepsilon_1; \mu_n').
$$

\end{lem}
\begin{proof}
The statement follows from considering $x_{n} \partial_n$ as an endomorphism of $X_{\mathfrak{sp}}(\mu)$ and decomposing  $X_{\mathfrak{sp}}(\mu)$  into a direct sum of eigenspaces of this endomorphism. 
\end{proof}

\section{The Lie algebras  $\mathfrak{sl}(\infty)$, $\mathfrak{o}(\infty)$, $\mathfrak{sp}(\infty)$}

In what follows we consider a fixed infinite chain of embeddings of simple Lie algebras
\begin{equation*}\frak g_1\hookrightarrow\frak g_2\hookrightarrow...\hookrightarrow\frak g_n\hookrightarrow\frak g_{n+1}\hookrightarrow...
\end{equation*}
such that $\textup{rk}\, \mathfrak{g}_n=n$. In addition, we assume that Cartan subalgebras $\mathfrak{h}_n$ of $\fg_n$ are fixed, and  the embeddings are root embeddings, i.e. $\gh_n$ is mapped into  $\gh_{n+1}$ and any $\gh_n$-root space of $\gg_n$ is mapped into a root space of $\gg_{n+1}$.  Then necessarily almost all $\fg_n$ are of one of the possible four types $\mathfrak{sl}({n+1})$, $\mathfrak{o}({2n+1})$ , $\mathfrak{o}({2n})$, $\mathfrak{sp}({2n})$.

 We define $\fg$ to be the direct limit Lie algebra $\fg=\underrightarrow{\lim}\,\fg_n$. Then, up to isomorphism, $\fg$ is one of the three Lie algebras $\mathfrak{sl}(\infty)$, $\mathfrak{o}(\infty)$, $\mathfrak{sp}(\infty)$, the Lie algebra $\mathfrak{o}({\infty})$ arises from both choices $\mathfrak{g}_n\simeq\mathfrak{o}({2n+1})$ and $\mathfrak{g}_n\simeq\mathfrak{o}({2n})$. This follows for instance form Baranov's classification \cite{B}.

By $\gh$ we denote the direct limit $\underrightarrow\lim\, \gh_i$. Then  $\gh$ is a maximal toral subalgebra which is \emph{splitting}, i.e. $\gg$ is a weight module over $\gh$. Such splitting maximal toral subalgebras are Cartan subalgebras according to the definition in \cite{DPS}. In what follows we refer to $\gh$ simply as \emph{Cartan subalgebra}, and will have $\gh$ fixed throughout the paper. 
We recall that, if  $\mathfrak{g}_n\simeq\mathfrak{sl}({\infty}),\,\mathfrak{sp}({\infty})$ there is only one $\Aut \fg$-conjugacy class of maximal toral subalgebras, and if $\fg\simeq\mathfrak{o}({\infty})$ there are two such $\Aut \fg$-conjugacy classes \cite{DPS}.
In the latter case we write $\fh=\fh_B$ if $\fg_n\simeq\mathfrak{o}({2n+1})$ and $\fh=\fh_D$ if $\fg_n\simeq\mathfrak{o}({2n})$. The Cartan subalgebras $\fh_B$ and $\fh_D$ are representatives of these two conjugacy classes.

In this paper we consider only \emph{splitting Borel subalgebras $\gb$ containing $\gh$}, that is, direct limits of Borel subalgebras of $\gg_n$ containing $\gh_n$. Equivalently, $\gb = \gh \oplus \bigoplus_{\alpha \in \Delta^+} \gg^{\alpha}$ for some triangular decomposition $\Delta = \Delta^+ \sqcup \Delta^-$, where $\Delta$ is the root system of $(\fg,\fh)$. Henceforth, we omit the adjective "splitting". We now consider the cases of $\mathfrak{sl}(\infty)$, $\mathfrak{o}(\infty)$, $\mathfrak{sp}(\infty)$ in detail.

For all three Lie algebras  there are well-defined linear functions $\varepsilon_i$ on $\fh$ which coincide with the linear functions $\varepsilon_i$ from Subsection \ref{subsec-general-sln}, when restricted to $\fh_n$ for every $n$. The weights $\lambda\in \fh^*$ are identified with formal  sums $\Sigma_{i=1}^\infty \lambda_i \varepsilon_i$, or with infinite sequences $(\lambda_1,\lambda_2,\dots)$. In the case of $\fg=\gs\gl(\infty)$, two infinite sequences determine the same weight if their difference is a constant sequence $(c,c,\dots)$.

Let $\gg = \mathfrak{sl} (\infty)$. The root system of $(\gg,\gh)$ is 
$$
\Delta_{\mathfrak{sl}(\infty)} = \{\varepsilon_i-\varepsilon_j \,|\, i\neq j,\; i,j\in \mathbb{Z}_{> 0}\}.
$$ 
The   Borel subalgebras $\gb$ containing $\fh$ is parameterized by arbitrary linear orders on $\mathbb Z_{> 0}$.  Given such an order $\prec$, the (positive) roots of the Borel subalgebra $\fb=\fb({\prec})$ are $\{\varepsilon_i-\varepsilon_j| \,i \prec j\}$.

Let now $\gg = \mathfrak{sp} (\infty)$. The root system of $(\gg, \gh)$ is
$$
\Delta_{\mathfrak{sp} (\infty)} = \{ \varepsilon_i - \varepsilon_j, \pm (\varepsilon_k + \varepsilon_{\ell}) \; | \; i \neq j, k \leq \ell, \; i,j,k, \ell\in \mathbb{Z}_{> 0} \}.
$$

The  Borel subalgebras $\gb$ of $\gg = \mathfrak{sp} (\infty)$ are parameterized by  pairs $(\prec, \sigma)$, where $\prec$ is a linear order of $\mathbb Z_{>0}$ and $\sigma : \mathbb Z_{>0} \to \{ \pm 1 \}$ is an arbitrary map. The  roots of the Borel subalgebra $\gb  (\prec, \sigma)$ are

$$\{\sigma(i)\varepsilon_i-\sigma(j)\varepsilon_j| i\prec j, \, i,j\in \mathbb{Z}_{> 0}\}\sqcup\{\sigma(i)\varepsilon_i+\sigma(j)\varepsilon_j| \, i,j\in \mathbb{Z}_{> 0}\}$$

Let $\gg = \mathfrak{o} (\infty)$.  Recall that there are two fixed  Cartan subalgebras of $\gg$: $\gh_B$ of type $B$, and $\gh_D$ of type $D$. The corresponding root systems are
$$\Delta_{B} = \{\pm \varepsilon_i, \pm \varepsilon_i \pm \varepsilon_j \; | \; i \neq j , \;  i,j\in \mathbb{Z}_{> 0}\}, \, \Delta_{D} = \{\pm \varepsilon_i \pm \varepsilon_j \; | \; i \neq j, \;  i,j\in \mathbb{Z}_{> 0} \}.$$

The  Borel subalgebras of $\gg$ containing $\gh_B$ or $\gh_D$ are parameterized by pairs $(\prec, \sigma)$, where $\prec$ is a linear order of $\mathbb Z_{>0}$ and $\sigma : \mathbb Z_{>0} \to \{ \pm1\}$ is an arbitrary map, satisfying the following condition in the case of $\Delta_D$:  if $\prec$ has a maximal element $i_0$ then $\sigma (i_0) = 1$. The roots of the Borel subalgebra $\fb(\prec,\sigma)$ are

$$\{\sigma(i)\varepsilon_i |  i\in \mathbb{Z}_{> 0}\}\sqcup\{\sigma(i)\varepsilon_i-\sigma(j)\varepsilon_j| i\prec j, \, i,j\in \mathbb{Z}_{> 0}\}\sqcup\{\sigma(i)\varepsilon_i+\sigma(j)\varepsilon_j| \, i\neq j,\, i,j\in \mathbb{Z}_{> 0}\}$$
for $\fh=\fh_B$, and are

$$\{\sigma(i)\varepsilon_i-\sigma(j)\varepsilon_j| i\prec j, \, i,j\in \mathbb{Z}_{> 0}\}\sqcup\{\sigma(i)\varepsilon_i+\sigma(j)\varepsilon_j|  \, i\neq j,\, i,j\in \mathbb{Z}_{> 0}\}$$
for $\fh=\fh_D$.

In all cases, $Q_{\fg}$ will denote the root lattice of the pair $(\fg,\fh)$. By $V$ we denote the \emph{natural} module of $\fg$. It is characterized by the  fact that  	\begin{equation*}
\Supp V = 
 \begin{cases}
   \{\varepsilon_i \, |i\in \mathbb{Z}_{>0} \}& \text{for} \,\, \fg=\mathfrak{sl}(\infty)  \\ 
   \{0,\pm\varepsilon_i  \, |i\in \mathbb{Z}_{>0} \}& \text{for} \,\, \fg=\mathfrak{o}(\infty), \, \fh=\fh_B\\
  \{\pm\varepsilon_i \, |i\in \mathbb{Z}_{>0} \} & \text{for} \,\, \fg=\mathfrak{o}(\infty), \, \fh=\fh_D\\
  \{\pm\varepsilon_i  \, |i\in \mathbb{Z}_{>0} \} & \text{for} \,\, \fg=\mathfrak{sp}(\infty).
 \end{cases}
\end{equation*}
For $\fg=\mathfrak{sl}(\infty)$, we also have the \emph{conatural} module $V_*$: $\Supp V_*=\{-\varepsilon_i| i\in \mathbb{Z}_{>0}\}$. Furthermore, note that $V\otimes V_*$ is an associative algebra; the associated Lie algebra is by definition the Lie algebra $\mathfrak{gl}(\infty)$. Clearly, $\fg=\mathfrak{sl}(\infty)$ can be identified with the kernel of the trace map $\mathfrak{gl}(\infty)\to\mathbb{C}$.

The notions of weight $\fg$-module, bounded weight module, multiplicity-free weight module, cuspidal weight module, and uniform weight module, are carried over verbatim from the case of $\fg_n$, see Subsection \ref{subsec-general-sln}. The same applies to the notions of support and essential support, denoted again by $\Supp (\cdot)$ and $\Supp _{\rm ess}(\cdot)$, respectively.

The enveloping algebra $U=U(\fg)$ is a weight $\fg$-module (with infinite-dimensional weight spaces) with respect to the adjoint action of $\fh:U=\bigoplus_{\beta \in Q_\fg} U^\beta$. The weight space $U^0$ is the centralizer of $\fh$ in $U$.

Since our fixed embeddings $\fg_n\hookrightarrow\fg_{n+1}$ are root embeddings, they induce natural monomorphisms on Weyl groups: $W_n\to W_{n+1}$. We set $W:=\underrightarrow\lim\, W_n$  and call $W$ the \emph{Weyl group of} $\fg$.

\section{General facts on bounded $\mathfrak{sl}(\infty)$-, $\mathfrak{sp}(\infty)$-, $\mathfrak{0}(\infty)$-modules}

Let $\gg = \mathfrak{sl} (\infty)$, $ \mathfrak{sp} (\infty)$, $ \mathfrak{o} (\infty)$. A $\fg$-module $M$ is \emph{locally simple}, if for any $0\neq m\in M$ there exists $n_0>0$ such that the $\fg_n$-module $U_n\cdot m$ is a simple $\fg_n$-module for $n>n_0$.

\begin{lem} \label{l1}
If  $M$ is a simple weight $\gg$-module and  $\lambda \in \Supp  M$, then $M$ is a simple $U^0$-module. 
\end{lem}
\begin{proof}

Note that any $m \in  M^{\lambda}$ generates $M$ over $U$ , i.e. $U \cdot  m = M$. Hence
$$
U\cdot m=\left( \bigoplus_{\beta  \in {Q_\fg}}  U^{\beta}\right)  \cdot m =  \bigoplus_{\gamma\,  \in  \Supp M} M^{\gamma}.
$$
that is 
$U^{\beta} \cdot m  = M^{\beta + \lambda}$ for every $\beta  \in {Q_\fg}$. Therefore $U^{\beta} \cdot M^{\lambda}  = M^{\beta + \lambda}$, and for $\beta =0$ every $m \in  M^{\lambda}$ generates $M^{\lambda}$ as a $U^0$-module, i.e. $M^{\lambda}$ is a simple $U^0$-module.
\end{proof}

\begin{lem} \label{lem_u_o}
Let $M$ be a simple weight $\gg$-module and let $\lambda \in {\rm Supp}\, M$. Then there exists $N>0$ such that $M^{\lambda}$ is a simple $U^0_n$-module for $n>N$. 
\end{lem}
\begin{proof} Note that by the previous lemma  $M^{\lambda}$ is a simple $U^0$-module.
Let $d_n$ be  the smallest dimension of an irreducible $U^0_n$-submodule of $M^{\lambda}$. The sequence $d_1,d_2,...$ grows monotonically, hence it  stabilizes at $\dim M^{\lambda}$ or at $d < \dim M^{\lambda}$. In the first case we are done. 

Assume we are in the second case.  Then let $M_n^{\lambda}$ be the  sum of all irreducible $U^0_n$-submodules of $M^{\lambda}$ of dimension $d$. The spaces $M_n^{\lambda}$ form a descending chain of subspaces of $M^{\lambda}$ of dimension bounded by $d$ from below, hence have a nontrivial intersection. This intersection is a   $U^0$-submodule  of $M^{\lambda}$, therefore it coincides with   $M^{\lambda}$ (as $M^{\lambda}$ is a simple $U^0$-module) and is necessarily simple. This means that $M^{\lambda} = M_n^{\lambda}$ for sufficiently large $n$, which completes the proof. \end{proof}

\begin{cor} \label{cor-loc-simple}
Every simple bounded weight $\gg$-module is locally simple.
\end{cor}
\begin{proof}
Let $M$ be an  infinite-dimensional simple bounded $\gg$-module and let $\lambda \in {\rm Supp}_{\rm ess}\, M$. Then by Lemma \ref{lem_u_o}  there is $N>0$ such that $M^{\lambda}$ is a simple $U^0_n$-module for $n>N$. Let $m \in M^{\lambda}$ and set $M_n = U_n \cdot m$.  The same argument as in \S\ref{subsec-general-sln} shows that the simplicity of $M$ implies that $M$ is uniform. Hence $M_n$ is uniform. We have two possible cases: $\dim M_n<\infty$ for all $n$, or $\textup{dim}M_n=\infty$ for almost all $n$. In the first case, for $n>N$, the simplicity of $M^\lambda$ as a $U^0_n$-module implies the simplicity of $M_n$ as a $\fg_n$-module. 

The second case is possible only for $\gg_n = \mathfrak{sl}(n+1), \mathfrak{sp}(2n)$ as there are no infinite-dimensional simple bounded modules of $\gg_n = \mathfrak{o}(2n+1),  \mathfrak{o}(2n)$, see Subsection \ref{subsec-notation}. In this case, Proposition \ref{prop_ext_simple} implies that all but finitely many $M_n$ are simple or all but finitely many $M_n$  are extensions of finite-dimensional by simple modules. To complete the proof, it is sufficient to consider the latter possibility. By Corollary 3.3, this could occur only for $\fg=\mathfrak{sl}(\infty)$. Assume that for each $n>N$ we have an exact sequence
\begin{eqnarray*}
0 \to M_n' \to M_n \to F_n \to 0,
\end{eqnarray*}
where $M_n'$ is an infinite-dimensional simple $\gg_n$-module and $F_n$ is a finite-dimensional $\gg_n$-module, possibly equal to $0$. Since $\Hom_{\gg_n} (M_n', F_{n+1})= 0$, we see that the monomorphism $M_n \to M_{n+1}$ induces a monomorphism  $M_n' \to M_{n+1}'$. Then the simplicity of $M$  implies  $M=\underrightarrow\lim M'_n$, which in turn shows that $M$ is locally simple.\end{proof}

\begin{prop} \label{prop-mult-free}
Every simple bounded nonintegrable $\gg$-module is multiplicity free.
\end{prop}
\begin{proof}
Let $M$ be a simple bounded nonintegrable $\gg$-module. By Corollary \ref{cor-loc-simple}, $M$ is a direct limit of simple bounded infinite-dimensional $\gg_n$-modules $M_n$. Let $d_n = \deg M_n$ and $d = \deg M$. Lemma \ref{lem_u_o} shows that, if $d>1$, then there is $N$ such that $d_n = d>1$ for $n>N$. By Lemma \ref{lemma-deg}, $d_n$ is greater than the dimension of a finite-dimensional simple nontrivial $\mathfrak{gl}({n-2})$-module. In particular, $d_n>n-2$, which is a contradiction.\end{proof}
We classify the nonintegrable simple bounded $\mathfrak{sl}(\infty)$-modules in \S \ref{subsec-int-sim-sl} below.
The following proposition is a preparatory result for the classification of integrable simple bounded $\mathfrak{sl}(\infty)$-modules. 

For the rest of the section we fix $\gg = \mathfrak{sl}(\infty)$. 

\begin{prop} \label{prop-integ-bounded}
Let $M$ be a simple integrable bounded $\gg$-module.
Then $M\simeq\underrightarrow\lim \,L_{\fg_n}(\lambda(n))$ where, for every $n$, $\lambda(n)$ is of one of the following five types:

\begin{itemize}
\item[(i)] $(1,1,...,1,0,0,...,0)$, where the number of 0's and 1's are both growing when $n\to \infty$;
\item[(ii)]  $(a_n,0,0,...,0)$,
\item[(iii)] $(0,0,...,0, -a_n)$,
\item[(iv)] $(\mu_1,...,\mu_k,0,0,...,0)$,
\item[(v)] $(0,0,...,0, - \mu_k,...,-\mu_1)$,
\end{itemize}
where $k$ and $\mu_1,...,\mu_k\in\mathbb{Z}_{>0}$ are fixed and such that $\mu_i - \mu_{i+1}  \in {\mathbb Z}_{\geq 0}$, and   $a_n \in {\mathbb Z}_{\geq 0}$ is a monotonic sequence with $\lim_{n\to \infty}a_n=\infty$.
\end{prop}

We will prove the proposition with the aid of five lemmas. We start with some notation. For a fixed $n$ and a dominant integral $\gg \gl(n+1)$-weight $\lambda = (\lambda_1,...,\lambda_{n+1})$, we put $d(\lambda) := \deg L_{\gg_{n}} (\lambda) $, $\mbox{lp}(\lambda) := \max \{ i \geq 1 \; | \; \lambda_1 = \cdots = \lambda_{i-1} = \lambda_i\}$ (the left part of $\lambda$), $\mbox{rp}(\lambda) := \max \{ i \geq 1 \; | \; \lambda_{n-i+2} =\cdots =\lambda_{n} = \lambda_{n+1}\}$ (the right part of $\lambda$). 

We also recall the Gelfand-Tsetlin decomposition rule for a simple finite-dimensional $\mathfrak{gl}(n)$-module $L_{\mathfrak{gl}({n+1})}(\lambda_1,...,\lambda_{n+1})$ considered as a $\left( \mathfrak{gl}({n}) \oplus \mathfrak{gl}(1)\right)$-module (here the roots of $\mathfrak{gl}(n) \oplus \mathfrak{gl}(1)$ are $\varepsilon_i-\varepsilon_j$, $1\leq i,j \leq n$):
\begin{equation} \label{gt-rule}
L_{\mathfrak{gl}({n+1})}(\lambda_1,...,\lambda_{n+1})|_{\left( \mathfrak{gl}({n}) \oplus \mathfrak{gl}(1)\right)} \simeq \bigoplus_{\mu \in {\rm GT} (\lambda)} L_{\mathfrak{gl}({n})}(\mu_1,...,\mu_{n}) \boxtimes L_{\mathfrak{gl}({1})}(|\lambda|-|\mu|), 
\end{equation}
where ${\rm GT}(\lambda)$ is the set of all $n$-tuples $(\mu_1,..., \mu_n)$ such that $\lambda_i - \mu_{i} \in {\mathbb Z}_{\geq 0}$, $\mu_i - \lambda_{i+1} \in {\mathbb Z}_{\geq 0}$, $i = 1,..,n$ \cite{Zh}.

\begin{lem} \label{lem-0}
Let $L_{\mathfrak{gl}({n+1})}(\lambda_1,...,\lambda_{n+1})$ be finite dimensional, and let $\mu',\mu'' \in {\rm GT} (\lambda)$ be such that $\mu' \neq \mu''$ and $|\mu'|=|\mu''|$. Then $d(\lambda) \geq d(\mu') + d(\mu'')$. 
\end{lem}
\noindent \emph{Proof of Lemma \ref{lem-0}.}  
Let $B(n-1) := \{\nu \in {\mathfrak h}_{n-1}^* \; | \;  |(\nu, \alpha)|\leq 1, \mbox{ for all } \alpha \in \Delta_{n-1} \} $ and  $B(\eta) := \Supp L(\eta)\cap B({n-1})$ for a dominant integral ${\mathfrak{gl}(n)}$-weight $\eta$. It is well known that $B(\eta) \neq \emptyset$ and that  $B(\eta)$ consist of a single orbit of the Weyl group $W_{n-1}$. We also have  
\begin{equation*} 
\deg L(\eta) = \max \{ \dim L(\eta)^{\nu} \; | \; \nu \in B(\eta)\}, 
\end{equation*}
which follows easily from  $\mathfrak{sl}(2)${-considerations}. Hence, $\deg L(\eta) = \dim L(\eta)^{\nu}$ for any $\nu \in B(\eta)$.

Now, since $\mu'$ and $\mu''$ satisfy the conditions of the lemma, we have $\mu'-\mu'' \in Q_{n-1}$ and hence $B(\mu') = B(\mu'')$. Furthermore,  by (\ref{gt-rule}),  the $\mathfrak{gl}({n})$-module $L(\lambda)$ has a submodule isomorphic to $L(\mu')\oplus L(\mu'')$. Choose $\nu = (\nu_1,...,\nu_{n}, \nu_{n+1}) \in \Supp L_{\mathfrak{gl}({n+1})}(\lambda)$ such that $(\nu_1,...,\nu_{n}) \in  B(\mu')$ and $\nu_{n+1} = |\lambda|-|\mu|$. Then 
\begin{eqnarray*}
 d(\lambda) \geq   \dim L(\lambda)^{\nu}  \geq  \dim L(\mu')^{(\nu_1,...,\nu_{n})} + \dim L(\mu'')^{(\nu_1,...,\nu_{n})}  = d(\mu') + d(\mu'').
 \end{eqnarray*}
\hfill $\square$

\begin{lem} \label{lem-1}
If $x, \ell \geq 2$, then $d((x,1,0^{(\ell-1)})) \geq \min \{ x, \ell\}$.
\end{lem}
\noindent\emph{Proof of Lemma \ref{lem-1}.}  Note that  the $\ell$-tuples $\mu':=(x-1,1,0^{(\ell -2)})$ and $\mu'':=(x,0,0^{(\ell -2)})$
 are in ${\rm GT} ((x,1,0^{(\ell-1)}))$ and $|\mu'|=|\mu''|=1$. Then using Lemma \ref{lem-0} and the obvious fact that $d(\eta) \geq 1$ for all $\eta \neq 0$, we have
$$
d((x,1,0^{(\ell-1)})) \geq d(\mu') + d(\mu'') \geq d(\mu') + 1.
$$
Induction on $\min \{ x, \ell\}$ completes the proof. \hfill $\square$

\begin{lem} \label{lem-2}
If $x\in \mathbb{Z}_{ \geq 2}$, then $d((x^{(k)},0^{(\ell)})) \geq \min \{ 1+(k-1)(x-1), \ell\}$ and $d((x^{(k)},0^{(\ell)})) \geq \min \{ 1+(\ell-1)(x-1), k\}$.
\end{lem}
\noindent\emph{Proof of Lemma \ref{lem-2}.}  By (\ref{gt-rule}), we have $d((x^{(k)},0^{(\ell)}))  \geq d((x^{(k-1)},1, 0^{(\ell-1)}))$.  Furthermore Lemma \ref{lem-0} yields
$$
d((x^{(k-1)},1, 0^{(\ell-1)})) \geq d((x^{(k-2)},x-1,1, 0^{(\ell-2)})) + d((x^{(k-1)}, 0, 0^{(\ell-1)})) \geq d((x^{(k-2)},x-1,1, 0^{(\ell-2)})) + 1.
$$

By continuing this line of reasoning, we get
\begin{eqnarray*}
d((x^{(k-1)},1, 0^{(\ell-1)})) & \geq & d((x^{(k-2)},x-1,1, 0^{(\ell-2)})) + 1 \geq \cdots \\
& \geq & d((1^{(k)}, 0^{(\ell-kx)})) + (k-1)(x-1) = 1 + (k-1)(x-1),
\end{eqnarray*}
under the assumption that $\ell>(k-1)(x-1)$. If $\ell \leq (k-1)(x-1)$ then 
$$
d((x^{(k-1)},1, 0^{(\ell-1)})) \geq d((x^{(k-2)},x-1,1, 0^{(\ell-2)})) + 1 \geq \cdots \geq \ell.
$$ 
To prove the second inequality we observe that $d((x^{(k)}, 0^{(\ell)})) = d((0^k, (-x)^k))$. Now we apply the same reasoning (but going from right to left) as above  for $(0^k, (-x)^k)$. \hfill $\square$

\begin{lem} \label{lem-3}
 Let $\lambda = (\lambda_1,...,\lambda_{n+1})$ be a dominant integral weight   such that $\lambda_1 - \lambda_{n+1} >1$. Then:
$$
d(\lambda) \geq \min\{ {n+1} - {\rm rp}(\lambda), {n+1} - {\rm lp}(\lambda)\}.
$$ 
\end{lem}
\noindent\emph{Proof of Lemma \ref{lem-3}.}  As usual, we apply again several times the Gelfand-Tsetlin rule (\ref{gt-rule}). We first observe that, if $i, j \geq 1$ are such that 
$$\lambda_i > \lambda_{i+1} \geq \lambda_{i+2} \geq \cdots \geq \lambda_{i+j} > \lambda_{i+j+1}
$$
then 
$$
d(\lambda) \geq d((\lambda_1,...,\lambda_i, \lambda_{i+2},...,\lambda_{i+j}, \lambda_{i+j+1}, ..., \lambda_{n+1})) + d((\lambda_1,...,\lambda_i -1, \lambda_{i+2},...,\lambda_{i+j}, \lambda_{i+j+1} + 1, ..., \lambda_{n+1})).
$$
Therefore
$$
d(\lambda) \geq d((\lambda_1,...,\lambda_i, \lambda_{i+2},...,\lambda_{i+j}, \lambda_{i+j+1}, ..., \lambda_{n+1})) + 1 \geq \cdots \geq d((\lambda_1,...,\lambda_i, \lambda_{i+j+1}, ..., \lambda_{n+1})) + j. 
$$
Let now $k =  \mbox{lp}(\lambda)$ and $\ell =  \mbox{rp}(\lambda)$. In particular: 
$$\lambda_1 = \cdots  = \lambda_k > \lambda_{k+1} \geq \cdots \geq \lambda_{n-\ell+1} > \lambda_{n-\ell + 2} = \cdots = \lambda_{n+1}.$$

If we apply the last inequality for $i = k$ and $j = n-k-\ell+1$, we obtain  
\begin{equation} \label{lem-ineq}
d(\lambda) \geq d((\lambda_1^{(k)}, 0^{(\ell)})) + n-k-\ell+1.
\end{equation}
But $\lambda_1 \geq 2$ since $\lambda_1 - \lambda_{n+1} >1$. So, by Lemma \ref{lem-2}, $d((\lambda_1^{(k)}, 0^{(\ell)})) \geq \min \{ k, \ell\}$ and that combined with (\ref{lem-ineq}) completes the proof of the lemma. \hfill $\square$

\begin{lem} \label{lem-4}
Let $k \geq 2$ and let $(\lambda_1,...,\lambda_k)$ be a dominant integral $\gg \gl (k)$-weight such that $\lambda_1 \geq 2$ and $\lambda_k>0$. Then 
$$
d((\lambda_1,...,\lambda_k, 0^{(\ell)})) \geq \min \{ \lambda_1, \ell\}.
$$ 
\end{lem}
\noindent\emph{Proof of Lemma \ref{lem-4}.} 
Applying $k-1$ times (\ref{gt-rule}), we obtain
$$
d((\lambda_1,...,\lambda_k, 0^{(\ell)})) \geq d((\lambda_1,...,\lambda_{k-1}, 1, 0^{(\ell-1)})) \geq d((\lambda_1, 1, 0^{(\ell-1)})).
$$
Now by Lemma \ref{lem-1}, $d((\lambda_1, 1, 0^{(\ell-1)})) \geq \min \{ \lambda_1, \ell\}$, which completes the proof.
\hfill $\square$
\begin{proof}[Proof of Proposition 4.5] Let $\lambda(n) = (\lambda_1,...,\lambda_{n+1})$. If $\lambda_1-\lambda_{n+1}>1$ and  $d((\lambda_1,\dots,\lambda_{n+1}))$ is bounded for $n\to\infty$, Lemma \ref{lem-3} shows that $(\lambda_1,\dots,\lambda_{n+1})$ can be written as $(\mu_1,\dots,\mu_k,0^{(n-k+1)})$ or $(0^{(n-k+1)},-\mu_k,\dots,-\mu_1)$ for some fixed $k$. Lemma \ref{lem-4} shows that if $\mu_k\neq 0$ for $k\geq 2$, then $\mu_1$ must be constant when $n\to \infty$. The Gelfand-Tsetlin rule then implies that $\mu_2,\dots, \mu_k$ are also constants, hence $\lambda(n)$ is of the form (iv) or (v). The case when $\mu_2=0$ leads to cases (ii) and (iii). Finally, the possibility $\lambda_1-\lambda_{n+1}=1$ yields case (i).
\end{proof}

\section{Classification of integrable simple bounded weight  $\mathfrak{sl} (\infty)$-, $\mathfrak{o}(\infty)$-, $\mathfrak{sp} (\infty)$-modules}

\subsection{The case of $\mathfrak{sl} (\infty)$} \label{subsec-int-sim-sl}

In this subsection $\gg = \mathfrak{sl} (\infty)$. We start with some definitions.

 A subset $A$ of $\Z_{>0}$ is \emph{semi-infinite} if both $A$ and $\Z_{>0} \setminus A$ are infinite. For two semi-infinite sets $A$ and $B$ of $\Z_{>0}$, we write $A \approx B$ if there exist disjoint finite subsets $F_A$ and $F_B$ of $A$ and $B$, respectively, so that $A \setminus F_A = B \setminus F_B$. Obviously, $\approx$ defines an equivalence relation on the set of semi-infinite subsets of $\Z_{>0}$.  

For a semi-infinite subset $A$ of $\Z_{>0}$, let $A_n:=A \cap [1,n]$ and $k_n = \# (A_n)$. Consider the $k_n$-th exterior power $\Lambda^{k_n} (V_{n+1})$ of the natural representation $V_{n+1}$ of $\mathfrak{sl}({n+1})$. Since $A$ is semi-infinite, we have $\Lambda^{k_n} (V_{n+1}) \neq 0$ for $n>>0$. Moreover, there is a unique, up to a multiplicative constant, monomorphism of $\gg_n$-modules $\Lambda^{k_n} (V_{n+1}) \hookrightarrow \Lambda^{k_{n+1}} (V_{n+2})$. The resulting direct limit  $\underrightarrow\lim\Lambda^{k_n} (V_{n+1})$ will be denoted by $\Lambda_A^{\frac{\infty}{2}}V$ and is a \emph{semi-infinite fundamental representation of} $\gg$. In other words, $\Lambda_A^{\frac{\infty}{2}}V$ is isomorphic to $\underrightarrow\lim L_{\gg_n} (1^{(k_n)}, 0^{(n+1-k_n)})$, cf. Proposition \ref{prop-integ-bounded}(i). It is straightforward to check that  $\Lambda_A^{\frac{\infty}{2}}V \simeq \Lambda_B^{\frac{\infty}{2}}V$ if and only if $A \approx B$.

 Given a Borel subalgebra $\fb(\prec)$, we say that $\fb(\prec)$ is a $A$-\emph{compatible} for a subset $A\subset \mathbb{Z}_{>0}$, if $a \prec b$ for all $a\in A$, $b\in \mathbb{Z}_{>0}\backslash A$.

Next, for any infinite subset $A=\{a_1, a_2, \dots |\, a_i<a_{i+1}\}$ of $\mathbb{Z}_{\geq 0}$, we introduce the representation $S_A^{\infty} V$ as follows. There is a unique, up to a multiplicative constant, monomorphism of $\gg_n$-modules $S^{a_n} (V_{n+1}) \hookrightarrow S^{a_{n+1}} (V_{n+2})$. The resulting direct limit   $\underrightarrow\lim \,S^{a_n} (V_{n+1})$ will be denoted by  $S_A^{\infty} V$. Note that $S_A^{\infty} V$ is isomorphic to $\underrightarrow\lim L_{\gg_n} (a_n,0^{(n-1)})$, cf. Proposition \ref{prop-integ-bounded}(ii). For two  infinite sets $A$ and $B$ as above, we write $A \sim B$ if $a_n=b_n$ for all $n$ greater than some $n_0$. 
It is straightforward to check  that $S_A^{\infty} V \simeq S_B^{\infty} V$ if and only if $A \sim B$. 

Similarly, for a sequence $A$ as above, we introduce the modules $S_A^{\infty} V_*$. These modules are isomorphic to  $\underrightarrow\lim L_{\gg_n} (0^{(n-1)}, -a_n)$, cf. Proposition \ref{prop-integ-bounded}(iii).

 Recall that a \emph{partition} is a tuple $\mu=(\mu_1,...,\mu_k)$ of positive integers satisfying $\mu_i\geq\mu_{i+1}$ for $i=1,...,k-1$. Given a partition $\mu$ we define the $\fg$-module $S^\mu V$ as a direct limit  $\underrightarrow\lim L_{\gg_n} (\mu_1,..,\mu_k,0^{(n+1-k)})$, cf. Proposition \ref{prop-integ-bounded}(iv). Clearly, $S^\mu V$ is well defined up to isomorphism. Similarly, we define the module $S^{\mu} V_*$ as  $\underrightarrow\lim L_{\gg_n} (0^{(n+1-k)}, -\mu_k,..,-\mu_1)$,  cf. Proposition \ref{prop-integ-bounded}(v).

\begin{thm} \label{thm-simple-integ-sl} Let $M$ be a simple integrable bounded weight module of $\gg = \gs \gl (\infty)$. Then $M$ is isomorphic to one of the following:
\begin{itemize}
\item[(i)] $\Lambda_A^{\frac{\infty}{2}}V$ for a semi-infinite set $A$,
\item[(ii)]  $S_A^{\infty}V$ for an infinite set $A$,
\item[(iii)] $S_A^{\infty}V_*$ for an infinite set $A$,
\item[(iv)]  $S^{\mu} V$ for a partition $\mu$,
\item[(v)]  $S^{\mu} V_*$ for a partition $\mu$.
\end{itemize}
All isomorphisms between modules from the above list are : $\Lambda_A^{\frac{\infty}{2}}V \simeq \Lambda_B^{\frac{\infty}{2}}V$ for $A\approx B$,  $S_A^{\infty}V \simeq S_B^{\infty}V$ and $S_A^{\infty}V_* \simeq S_B^{\infty}V_*$ for $A\sim B$, $S^\emptyset V= S^\emptyset V_*=\mathbb{C}$ ($\emptyset$ stands for the empty partition).
\end{thm} 

\begin{proof}
The statement is a direct consequence of Proposition \ref{prop-integ-bounded} as in each case the modules listed in the theorem  account, up to isomorphism, for all possible direct limit modules from Proposition \ref{prop-integ-bounded}. The claim about isomorphisms is straightforward.
\end{proof}

Recall that if $\fb\supset \fh$ is a Borel subalgebra, a $\fg$-module $M$ is a $\fb$\emph{-highest weight module} if it is generated by a vector $m$ such that $m\in M^{\lambda}$ for some $\lambda \in \fh^*$ and $\fg^\alpha\cdot m=0$ for any root $\alpha$ of $\fb$. 

\begin{prop} \label{5.2}
\begin{itemize}
\item[(a)] The modules  $\Lambda_A^{\frac{\infty}{2}}V$,  $S_A^{\infty}V$ and $S_A^{\infty}V_*$ are multiplicity free.
\item[(b)] The following holds:
\begin{itemize}
\item[(i)] $\Supp \Lambda_A^{\frac{\infty}{2}}V = \{  \sum_{i \in B} \varepsilon_i  \; | \; B \approx A\}$,
\item[(ii)]  $\Supp S_A^{\infty}V =  \{ \lambda   \; | \; \lambda_i\geq0, \, \exists n: \sum^n_{i=1}\lambda_i=a_n, \lambda_i = a_i-a_{i-1} \textup{ for } i>n \}$,
\item[(iii)]  $\Supp S_A^{\infty}V_* =  \{  \lambda   \; | \; \lambda_i\leq0, \, \exists n: \sum^n_{i=1}\lambda_i=-a_n, \lambda_i = a_{i-1}-a_i \textup{ for } i>n \}$,
\item[(iv)] $\Supp S^\mu V= \{\lambda \;  |  \; 0\leq \lambda_i \leq\mu_i \}$ where $\mu=(\mu_1,\dots, \mu_k)$,
\item[(v)] $\Supp S^\mu V_*= \{ \lambda \; | \; 0\leq -\lambda_i \leq\mu_i \}$ where $\mu=(\mu_1,\dots \mu_k)$.

\end{itemize}

\item[(c)] The module $ \Lambda_A^{\frac{\infty}{2}}V$ is $\mathfrak b$-highest weight module  for a given Borel subalgebra $\mathfrak b\supset \fh$, if and only if there is a semi-infinite set $B\subset\mathbb{Z}_{>0}$ such that $B\approx A$ and $\fb$ is $B$-compatible.

\item[(d)] The modules  $S^{\infty}_A V$ and $S^{\infty}_A V_*$ are not highest weight modules (for any Borel subalgebra $\fb\supset\fh$).

\item[(e)]  If $\mu=(\mu_1,\dots,\mu_k)$, the module $S^{\mu}V$ (respectively, $S^{\mu}V_{*}$) is a $\fb$-highest weight module for any Borel subalgebra $\fb\supset\fh$ whose corresponding order $\prec$ has the property that there are $k$ indices $i_1,\dots,i_k$ satisfying $i_1\prec\dots\prec i_k$, and $i_k\prec i$ (respectively, $i_1\succ \dots \succ i_k$ and $i_k\succ i$) $\forall i\neq i_1,\dots, i_k$.

\end{itemize}
\end{prop}

\begin{proof}
Straightforward.
\end{proof}

\subsection{The case of $\mathfrak{o} (\infty)$} In this subsection $\fg=\go(\infty)$, and either $\fg_n=\mathfrak{o}(2n+1)$ for any $n$ or  $\fg_n= \mathfrak{o}(2n)$ for any $n$.
Recall that there are no infinite-dimensional simple bounded $\gg_n$-modules for $n\geq 3$, see Subsection \ref{subsec-general-sln}. Therefore, every bounded  simple weight $\fg$-module is integrable. 

\begin{prop}\label{asd} Let $M$ be a nontrivial bounded simple integrable weight $\fg$-module $M$. Then $M$ is  multiplicity free and is  isomorphic to $V$ or to a direct limit of (simple) spinor modules of $\fg_n$.
\end{prop}

\begin{proof} Consider the standard embedding $\mathfrak{gl}(\infty)\subset\fg$ such that $\fh\subset\mathfrak{gl}(\infty)$. Here $\fh=\fh_B$ or $\fh_D$. Restricting the weights in $\Supp M$ to $\fh\cap\mathfrak{sl}(\infty)$ is an injective map as the simplicity of $M$ as $\fg$-module implies that no two weights in $\Supp M$ differ by a constant weight $c\sum^\infty_{i=1}\varepsilon_i$ for $c\neq 0$. Consequently, $M_{|\mathfrak{sl}(\infty)}$ is a bounded integrable  weight $\mathfrak{sl}(\infty)$-module, and is  semisimple by Proposition 3.3 in \cite{PS}. Therefore, by Theorem \ref{thm-simple-integ-sl}, $M_{|\mathfrak{sl}(\infty)}$ is isomorphic to a direct sum of modules as in Theorem \ref{thm-simple-integ-sl}(i)--(v). The $\mathfrak{sl}(\infty)$-weights determine the $\fg$-weights up to a constant, and $\Supp M$ is $W$-stable. Since $W$ contains sign changes, a direct inspection of the $\mathfrak{sl}(\infty)$-supports of the modules from Theorem \ref{thm-simple-integ-sl} (see Proposition \ref{5.2}(b)) shows that we have two cases: $\Supp M$=$\Supp V$ or $\lambda\in \Supp M\Rightarrow \lambda_i=\pm \frac{1}{2}$. In the first case, Lemma \ref{l1} shows that $M\simeq V$. In the second case, we conclude that $M_{|\fg_n}$ is a direct sum of spinor modules. Consequently, $M$ is isomorphic to a direct limit of finite-dimensional $\fg_n$-modules which are  direct sums of spinor modules. Finally, Lemma \ref{lem_u_o} shows that $M$ is isomorphic to a direct limit of  spinor modules. Consequently, $M$ is  multiplicity free.
\end{proof}

We fix the Borel subalgebra $\gb_n$ with positive roots $\{ \varepsilon_i - \varepsilon_j, \varepsilon_k + \varepsilon_{\ell}, \varepsilon_m \; | \; i<j, k< \ell\}$ if  $\fg_n=\mathfrak{o}(2n+1)$, and $\{ \varepsilon_i - \varepsilon_j, \varepsilon_k + \varepsilon_{\ell} \; | \; i<j, k< \ell\}$ if  $\fg_n=\mathfrak{o}(2n)$. Recall that the Lie algebra $\fg_n=\mathfrak{o}(2n)$ has, up to isomorphism, two spinor representations, ${\mathcal S}_n^+$ and ${\mathcal S}_n^-$, while the  Lie algebra $\fg_n=\mathfrak{o}(2n+1)$ has one spinor representation, ${\mathcal S}_n$. These spinor modules have $\gb_n$-highest weights $\left(\frac{1}{2},..., \frac{1}{2},\frac{1}{2}\right)$ and $\left(\frac{1}{2},..., \frac{1}{2},-\frac{1}{2}\right)$ for $\fg_n=\mathfrak{o}(2n)$  and $\left(\frac{1}{2},..., \frac{1}{2},\frac{1}{2}\right)$  for $\fg_n=\mathfrak{o}(2n+1)$. It is a standard fact
 that, up to a multiplicative constant, there are exactly two weight embeddings of ${\mathcal S}_n$ into ${\mathcal S}_{n+1}$, which will be denoted by $\iota_n^+$ and $\iota_n^-$. In one of them vectors of weight $\left( \frac{1}{2}\right)^{(n)}$ are mapped to vectors of weight $\left( \frac{1}{2}\right)^{(n+1)}$, in the other -- to vectors of weight  $\left(  \left( \frac{1}{2}\right)^{(n)}, -\frac{1}{2}\right)$. Furthermore, it is easy to check that there are unique, up to multiplicative constants, weight embeddings  ${\mathcal S}_n^+ \hookrightarrow  {\mathcal S}_{n+1}^+$, ${\mathcal S}_n^+ \hookrightarrow  {\mathcal S}_{n+1}^-$, ${\mathcal S}_n^- \hookrightarrow  {\mathcal S}_{n+1}^+$  and  ${\mathcal S}_n^- \hookrightarrow  {\mathcal S}_{n+1}^-$. 

We now define the spinor weight $\mathfrak{o}(\infty)$-modules  of type $B$ and type $D$. 
\begin{defn} Let $A$ be a subset of $\mathbb{Z}_{>0}$.
\begin{itemize}
\item[(i)] By ${\mathcal S}_A^B$ we denote the $\fh_B$-weight $\gg$-module obtained as the direct limit of  the weight embeddings $\varphi_n : {\mathcal S}_n \to {\mathcal S}_{n+1}$ defined  by  $\varphi_n = \iota_n^+$ if $n \in A$  and $\varphi_n = \iota_n^-$ if $n\notin A$.
\item[(ii)] By ${\mathcal S}_A^D$ we denote the $\fh_D$-weight  $\gg$-module obtained as the direct limit of  weight embeddings $\varphi_n : M_n \to M_{n+1}$ such that $M_n = {\mathcal S}_n^+$ if $n \in A$ and $M_n = {\mathcal S}_n^-$ if $n \notin A$.
\end{itemize}
\end{defn}

To a subset $A\subset \mathbb{Z}_{>0}$  we assign weights  $\omega_A^B \in \fh_B^*$ and  $\omega_A^D \in \fh_D^*$ by putting $(\omega_A^B)_k = (\omega_A^D)_k := \frac{1}{2}$ if $k \in A$, $(\omega_A^B)_k = (\omega_A^D)_k := -\frac{1}{2}$ otherwise. For two  weights $\lambda,\mu \in \fh_B^*$  we write $\lambda\sim_{B} \mu$ if $\lambda_i = \mu_i$ for all but finitely many indices $i>0$.
 For two weights $\lambda,\mu\in \fh_D^*$  we write $\lambda\sim_{D} \mu$ if $\lambda_i \neq \mu_i$ for an even (in particular, finite) number of  indices $i>0$.

The following theorem, together with Proposition \ref{asd}, completes the classification of simple bounded $\mathfrak{o}(\infty)$-modules.

\begin{thm}\label{th1} Let $M$ be a simple weight module which is a direct limit of spinor modules of $\mathfrak{o}(2n+1)$ or $\mathfrak{o}(2n)$ for $n\to \infty$. Then $M\simeq \mathcal S_A^B$ or $M\simeq \mathcal S_A^D$ for some $A\subset\mathbb{Z}_{>0}$.
 Furthermore,  ${\mathcal S}_{A_1}^B \simeq {\mathcal S}_{A_2}^B$ if and only if $A_1 \sim A_2$. In addition, the following holds:
\begin{itemize}
\item[(i)] The weight modules ${\mathcal S}_A^B$ and ${\mathcal S}_A^D$ are multiplicity free.
\item[(ii)] $\Supp {\mathcal S}_A^B = \{\lambda \in \fh_B^*\; | \; \lambda_i = \pm \frac{1}{2}, \lambda \sim_B \omega_A^B\}$, \, 
$\Supp {\mathcal S}_A^D = \{\lambda \in \fh_D^*\; | \; \lambda_i = \pm \frac{1}{2}, \lambda \sim_D \omega_A^D\}$. 
\item[(iii)] A weight $\lambda$ in the support of ${\mathcal S}_A^B$ or ${\mathcal S}_A^D$ is a highest weight  relative to a Borel subalgebra $\gb(\prec, \sigma)$ if and only if $\sigma(i) = \pm 1$ precisely when $\lambda_i =  \pm \frac{1}{2}$. In particular, the modules $\mathcal{S}_A^B$ and $\mathcal{S}_A^D$ are highest weight modules.
\end{itemize}
\end{thm}
\begin{proof}
 It is clear that $\mathcal S_A^B$ and $\mathcal S_A^D$ account, up to isomorphism, for all direct limits of spinor modules, therefore the statement that $M\simeq \mathcal S_A^B$ or $M\simeq \mathcal S_A^D$ is straightforward. All other statements are also straightforward, and we leave the details to the reader.
\end{proof}

\begin{rmk}If one considers the subalgebra $\tilde\fg=\mathfrak{o}(\infty)$ of $\fg=\mathfrak{o}(\infty)$  which contains a Cartan subalgebra $\fh_B$ of $\fg$ and the root spaces of all long roots of $\fg$, then one can check that
$$
{\mathcal S}_A^B |_{\tilde \fg} \simeq  {\mathcal S}_A^D \oplus {\mathcal S}_{{\mathbb Z}_{>0} \setminus A}^D.
$$ 
\end{rmk}

\subsection{The case of $\mathfrak{sp} (\infty)$}

\begin{prop} \label{pro2}
Let $M$ be a nontrivial bounded simple integrable  weight $\mathfrak{sp}(\infty)$-module. Then $M$ is isomorphic to the natural $\mathfrak{sp}(\infty)$-module $V$.
\end{prop}

\begin{proof}
The statement follows by the same argument as in the proof of Proposition \ref{asd}. The only difference is that the case  $\lambda_i=\pm \frac{1}{2}$ is excluded
as such weights are not integral for $\mathfrak{sp}(\infty)$.
\end{proof}
\subsection{Simple minuscule modules}

 As in the finite-dimensional case, we call a  simple weight $\mathfrak{g}$-module $M$ \emph{minuscule} if the Weyl group $W$ acts transitively on $\Supp M$.

\begin{cor} Up to isomorphism, the minuscule simple bouded weight $\fg$-modules are precisely the integrable multiplicity-free simple weight $\fg$-modules, that is, the modules $\Lambda_A^{\frac{\infty}{2}}V$, $S_A^{\infty} V$, $S_A^{\infty} V_*$, $V$, $V_*$ for $\fg=\mathfrak{sl}(\infty)$, the modules $V$ and $\mathcal S_A^B$ and $\mathcal S_A^D$ for $\fg=\mathfrak{o}(\infty)$, and the module $V$ for $\fg=\mathfrak{sp}(\infty)$.
\end{cor}

\begin{proof}An inspection of the nonintegrable simple weight $\fg$-modules classified in Section \ref{sec-class-nonint} below implies that  none of them is minuscule. A further inspection of the supports of integrable simple weight modules from Theorem \ref{thm-simple-integ-sl}, Theorem \ref{th1}  and Proposition \ref{pro2}   verifies our claim.
\end{proof}

\section{Classification of nonintegrable simple bounded weight $\mathfrak{sl} (\infty)$-, $\mathfrak{sp} (\infty)$-modules} \label{sec-class-nonint}

\subsection{Families of multiplicity-free  $\mathfrak{sl} (\infty)$- and $\mathfrak{sp} (\infty)$-modules}

By ${\mathcal D} (\infty)$ we denote the algebra of polynomial differential operators of $\C [x_1,x_2,...]$ where now $x_1$, $x_2$ $\dots$ are infinitely many formal variables. The algebra ${\mathcal D} (\infty)$ is generated by $x_i, \partial_i$ for $i\in \mathbb{Z}_{>0}$, subject to the usual relations. Furthermore, $\{ x^{\alpha} \partial^{\beta} \; | \; \alpha, \beta \in ({\mathbb Z}_{\geq 0})_{\rm f}^{{\mathbb Z}_{ > 0}}\}$ forms a basis of ${\mathcal D} (\infty)$ where $x^{\alpha}: = x_1^{\alpha_1}x_2^{\alpha_2}...$ and $\partial^{\beta}: = \partial_1^{\beta_1}\partial_2^{\beta_2}...$.

A  ${\mathcal D} (\infty)$-module $M$ is a weight  ${\mathcal D} (\infty)$-module if 
$$
M = \bigoplus_{{\bf a} \in {\C}^{{\mathbb Z}_{>0}}} M^{\bf a},\;  \dim M^{\bf a} < \infty,
$$
where $M^{\bf a} = \{ m \in M \; | \; x_i \partial_i \cdot m = a_i m, \mbox{ for all }i \geq 1\}$. Recall that  $e_{ij}  \in \gg^{\varepsilon_i - \varepsilon_j}$. The map $e_{ij}\mapsto x_i \partial_j$ extends to a homomorphism $U(\gg) \to {\mathcal D} (\infty)$, which defines a functor from the category of  ${\mathcal D} (\infty)$-modules to the category of  $\mathfrak{sl} (\infty)$-modules. We first introduce the $\mathfrak{sl} (\infty)$-analogs of the modules introduced in Definition \ref{def-f}.

Let $\boldsymbol{\mu} \in \C^{{\mathbb Z}_{>0}}$. Consider the ${\mathcal D} (\infty)$-module
$$
{\mathcal F}_{\mathfrak{sl}} (\boldsymbol{\mu} ):= \{ { x}^{\boldsymbol{\mu}} p \; | \; p \in \C [x_1^{\pm 1}, x_2^{\pm 1}, \dots], \deg p = 0\}.
$$
Then ${\mathcal F}_{\mathfrak{sl}} (\boldsymbol{\mu} )$ is a bounded weight module over $\gg \gl (\infty)$ and $\gs \gl (\infty)$.
Furthermore, we introduce $V_{\mathfrak{sl}} (\boldsymbol{\mu})$ and $V_{\mathfrak{sl}} (\boldsymbol{\mu})^+$ analogously to the finite-dimensional case (see Subsection \ref{subsec-sln-mult-free}).  Namely,  we set
$$
V_{\mathfrak{sl}} (\boldsymbol{\mu}) := {\rm Span} \{ { x}^{\boldsymbol{\lambda}} \; | \; \boldsymbol{\lambda} - \boldsymbol{\mu} \in Q_{\mathfrak{gl} (\infty)}, {\rm Int}^+ (\boldsymbol{\mu}) \subset  {\rm Int}^+ (\boldsymbol{\lambda}) \}.
$$
Like in the case of $\mathfrak{sl}({n+1})$, we have  $V_{\mathfrak{sl}} (\boldsymbol{\mu'}) \subset V_{\mathfrak{sl}} (\boldsymbol{\mu})$ if and only if $\boldsymbol{\mu} - \boldsymbol{\mu'} \in Q_{\mathfrak{gl} (\infty)}$ and ${\rm Int}^+ (\boldsymbol{\mu}) \subset  {\rm Int}^+ (\boldsymbol{\mu'})$. 

\begin{defn}  \label{defn-infinite} Let  $V_{\mathfrak{sl}} (\boldsymbol{\mu})^+ := 0$ whenever ${\rm Int}^+ (\boldsymbol{\mu}) = {\rm Int} (\boldsymbol{\mu})$. If ${\rm Int}^+ (\boldsymbol{\mu}) \subsetneq {\rm Int} (\boldsymbol{\mu})$  define
$$
V_{\mathfrak{sl}} (\boldsymbol{\mu})^+ :=\sum_{V_{\mathfrak{sl}} (\boldsymbol{\mu'}) \subsetneq V_{\mathfrak{sl}} (\boldsymbol{\mu})} V_{\mathfrak{sl}} (\boldsymbol{\mu'}). 
$$
Then set $X_{\mathfrak{sl}}(\mu) := V_{\mathfrak{sl}} (\mu)/ V_{\mathfrak{sl}} (\mu)^+$.
\end{defn}
 Note that  $X_{\mathfrak{sl}}(\mu)$ is both a $\mathfrak{gl}(\infty)$-module and an $\mathfrak{sl}(\infty)$-module. 
 
 In the case of  $\mathfrak{sp}(\infty)$, for $\mu = (\mu_1,\mu_2,...)$ we introduce  modules $X_{\mathfrak{sp}}(\mu)$ similarly to the finite-dimensional case.  The difference between the notations $X_{\mathfrak{sl}}(\mu)$ and $X_{\mathfrak{sp}}(\mu)$ which we just introduced and the respective notations introduced from  Section \ref{sec-background} is that $\mu$  is an infinite sequence for $\gg = \mathfrak{sl}(\infty), \mathfrak{sp}(\infty)$.
 
\begin{defn}
Let $\mu$ and $\mu'$ are in $\C^{{\mathbb Z}_{>0}}$.   We write $\mu \sim_{\mathfrak{sl}} \mu'$ (respectively, $\mu \sim_{\mathfrak{sp}} \mu'$) if  $\mu - \mu' \in Q_{\mathfrak{gl}(\infty)}$ (respectively, $ \mu - \mu' \in Q_{\mathfrak{sp}(\infty)}$) and ${\rm Int}^+ (\mu) = {\rm Int}^+ (\mu')$.\end{defn}
 Note that  $\mu \sim_{\mathfrak{sl}} \mu'$ (respectively, $\mu \sim_{\mathfrak{sp}} \mu'$) implies ${\rm Int}^- (\mu) = {\rm Int}^- (\mu')$. 

Like in Subsection \ref{subsec-sln-mult-free}, when we consider an infinite sequence as a weight we automatically take the projection into $\gh^*$.

\subsection{Main results}

\subsubsection{The case  $\gg = \mathfrak{sl} (\infty)$}
\begin{thm} \label{thm-bounded-infinite-sl}

\begin{itemize} 
\item[(i)]  Every simple bounded nonintegrable $\mathfrak{sl} (\infty)$-module $M$ is isomorphic  to $X_{\mathfrak{sl}}(\boldsymbol{\mu})$ for some $\mu \in \C^{{\mathbb Z}_{>0}}$. In particular,  $M\simeq  \underrightarrow\lim X_{\mathfrak{sl}}(\boldsymbol{\mu}^n)$, where $\boldsymbol{\mu}^n = (\mu_1,...,\mu_n)$.

\item[(ii)] 
$X_{\mathfrak{sl}}(\boldsymbol{\mu}) \simeq X_{\mathfrak{sl}}(\boldsymbol{\mu}')$  if and only if $\boldsymbol{\mu} \sim_{\mathfrak{sl}} \boldsymbol{\mu'}$ or $\{\boldsymbol{\mu}, \, \boldsymbol{\mu}' \} = \{ 0^{(\infty)}, (-1)^{(\infty)}\}$. 

\item[(iii)]  The root space $\gg^{\varepsilon_i - \varepsilon_j}$ acts locally finitely on $X_{\mathfrak{sl}}(\boldsymbol{\mu})$ if and only if $i \in {\rm Int}^- (\boldsymbol{\mu})$ or $j \in {\rm Int}^+ (\boldsymbol{\mu})$.  In particular, $X_{\mathfrak{sl}}(\mu)$ is integrable if and only if ${\rm Int}^- (\mu) = \Z_{>0}$ or ${\rm Int}^+ (\mu) =  \Z_{>0}$, while  $X_{\mathfrak{sl}}(\nu)$ is cuspidal if and only if $ {\rm Int} (\nu) = \emptyset$, and in that case $X_{\mathfrak{sl}}(\nu) = {\mathcal F}_{\mathfrak{sl}} (\nu)$. Furthermore, $X_{\mathfrak{sl}}(\mu) \simeq \C$ (equivalently, $X_{\mathfrak{sl}}(\mu)$ is finite dimensional) if and only if $\mu  = 0^{(\infty)}$ or  $\mu  = (-1)^{(\infty)}$.

\item[(iv)] $\Supp X_{\mathfrak{sl}}(\boldsymbol{\mu}) = \{\lambda \; | \; \boldsymbol{\lambda} \sim_{\mathfrak{sl}}  \boldsymbol{\mu} \}$.
\end{itemize}
\end{thm}
\begin{proof}
 We prove (i), the remaining parts follow by (i) and the corresponding statements for the $\gg_n$-modules $X_{\mathfrak{sl}}(\boldsymbol{\mu}^{n+1})$, see Proposition \ref{prop-x-sl}(i),(ii) and Lemma \ref{lem-x-sl-iso}(ii). 

 By Proposition \ref{prop-mult-free}, $M$ is multiplicity free.  Following the proof of Corollary \ref{cor-loc-simple}, we fix $\lambda \in \C^{\Z_{>0}}$ such that  $\lambda \in \Supp M$, $m \in M^{\lambda}$ and $N>0$ so that $M_n = U(\gg_n) \cdot m$ is simple for $n \geq N$. Then all $M_n$, $n \geq N$, are simple infinite-dimensional multiplicity-free modules.  By Theorem \ref{thm-bounded-finite}, $M_N \simeq X_{\mathfrak{sl}}(\nu)$ for some $\nu \in \C^{N+1}$, while  by Proposition \ref{prop-x-sl}(iii) the central character of  $M_N$ is of the form $\chi_{c_N \varepsilon_1 + \rho}$ for some $c_N \in \C$. Corollary \ref{cor-x-sl} implies $M_N \simeq  X_{\mathfrak{sl}} (\mu_1,...,\mu_{N+1})$ where $\mu_i = \lambda_i + x$ and $x = \frac{1}{N+1} \left( c_N - \lambda_1-...- \lambda_{N+1}\right)$. Since $(\lambda_1,...,\lambda_{N+2}) \in \Supp M_{N+1}$, by Proposition \ref{prop-x-sl}(ii), we have  $M_{N+1} \simeq  X_{\mathfrak{sl}} (\mu_1+y,...,\mu_{N+1}+y,\mu_{N+2})$ for some $y,\mu_{N+2} \in \C$. Since $X_{\mathfrak{sl}} (\mu_1,...,\mu_{N+1})$ is isomorphic to a submodule of $X_{\mathfrak{sl}} (\mu_1+y,...,\mu_{N+1}+y,\mu_{N+2}) |_{\mathfrak{sl}(N+1)}$, by Lemma \ref{lem-decomp} we conclude that 
$(\mu_1+y,...,\mu_{N+1}+y,\mu_{N+2}) \sim_{\mathfrak{sl}} (\mu_1,...,\mu_{N+1},\mu_{N+2} + (N+1)y)$. Hence we may assume that $y = 0$ and $M_{N+1} \simeq  X_{\mathfrak{sl}} (\mu_1,...,\mu_{N+1}, \mu_{N+2})$. Proceeding the same way, we conclude that there are $\mu_{N+3}, \mu_{N+4}$,... so that $M_{N+k} \simeq  X_{\mathfrak{sl}} (\mu_1,..., \mu_{N+k+1})$, $k \geq 1$.

 By Corollary \ref{cor-mult-free-decomp} there is a unique (up to a constant multiple) monomorphism of $\mathfrak{sl} (n)$-modules $X_{\mathfrak{sl}}(\mu_1,...,\mu_n) \to X_{\mathfrak{sl}}(\mu_1,...,\mu_n, \mu_{n+1})$. Let $\boldsymbol{\mu} = (\mu_1,\mu_2,...)$ and $M' = X_{\mathfrak{sl}}(\boldsymbol{\mu})$. It is easy to check that if $\lambda' \in \Supp M'$, $m' \in (M')^{\lambda'}$, and $M'_n = U(\gg_n)\cdot m'$, then $M_n' \simeq X_{\mathfrak{sl}} (\mu_1,...,\mu_{n+1})$. Hence, $M \simeq M'$.\end{proof}

\begin{rmk} If we consider $X_{\mathfrak{sl}} (\mu)$ as a $\mathfrak{gl} (\infty)$-module, the support  changes as the weights of  $X_{\mathfrak{sl}} (\mu)$  are elements in $\C^{\Z_{>0}}$. Denote the $\mathfrak{gl}(\infty)$-support of  $X_{\mathfrak{sl}} (\mu)$  by $\Supp_{\mathfrak{gl}}X_{\mathfrak{sl}} (\mu)$. Then one can show that there is an isomorphism of $\mathfrak{gl} (\infty)$-modules $X_{\mathfrak{sl}} (\mu) \simeq X_{\mathfrak{sl}} (\mu')$  if and only if $\boldsymbol{\mu} \sim_{\mathfrak{sl}} \boldsymbol{\mu'}$. Thus $\Supp_{\mathfrak{gl}}X_{\mathfrak{sl}} (\mu)$ determines $X_{\mathfrak{sl}} (\mu)$  up to isomorphism. Consequently, the invariant which distinguishes two $\mathfrak{sl} (\infty)$-modules $X_{\mathfrak{sl}} (\mu)$ and $X_{\mathfrak{sl}} (\mu')$ with coinciding supports is their  respective $\mathfrak{gl} (\infty)$-module structure.
\end{rmk}

\begin{cor} 
\begin{itemize}
\item[(i)] If  $X_{\mathfrak{sl}} (\boldsymbol{\mu})$ is integrable, then one of the following holds.
\begin{itemize}
\item[(a)] $X_{\mathfrak{sl}} (\boldsymbol{\mu}) \simeq S^{m} (V)$ or $X_{\mathfrak{sl}} (\boldsymbol{\mu}) \simeq S^m (V_*)$ for some $m \geq 0$, 
\item[(b)]$X_{\mathfrak{sl}} (\boldsymbol{\mu}) \simeq S^{\infty}_A V$ or $X_{\mathfrak{sl}} (\boldsymbol{\mu}) \simeq S^{\infty}_A V_*$ for some infinite set $A$.
\end{itemize}
\item[(ii)] The simple integrable bounded $\mathfrak{sl} (\infty)$-modules which are not isomorphic to  $X_{\mathfrak{sl}} (\boldsymbol{\mu})$ for any $\mu$ (up to isomorphism) are $\Lambda_A^{\frac{\infty}{2}} V$, $S^{\nu} V$, $S^{\nu}V_*$, for $\nu = (\nu_1,...,\nu_k)$, $k>1$.
\end{itemize}
\end{cor}
\begin{proof} (i). By Theorem \ref{thm-bounded-infinite-sl}(iii), $X_{\mathfrak{sl}} (\boldsymbol{\mu})$ is integrable if ${\rm Int}^+ (\mu) =  \Z_{>0}$ or ${\rm Int}^- (\mu) = \Z_{>0}$. If  ${\rm Int}^+ (\mu) =  \Z_{>0}$ and $\mu_i = 0$ for all but finitely many indices $i$, then $X_{\mathfrak{sl}} (\boldsymbol{\mu}) \simeq S^m (V)$ for $m = \sum_{i>0} \mu_i$.  If  ${\rm Int}^+ (\mu) =  \Z_{>0}$ and infinitely many $\mu_i$ are positive, then $X_{\mathfrak{sl}} (\mu) \simeq S^{\infty}_AV$ for the infinite set $A = \{a_1 = \mu_1,..., a_i = \mu_1+\cdots +\mu_i,... \}$. In the case ${\rm Int}^- (\mu) = \Z_{>0}$, we consider two subcases:  $\mu_i = -1$ for all but finitely many indices $i$, or  $\mu_i<-1$ for infinitely many indices $i$. In the first subcase $X_{\mathfrak{sl}} (\mu) \simeq S^m (V_*)$, $m = \sum_{i>0} (-1 - \mu_i)$. In the second subcase, $X_{\mathfrak{sl}} (\mu) \simeq S^{\infty}_AV_*$ for  $A = \{a_1 = -1 - \mu_1,..., a_i = -i - \mu_1-\cdots - \mu_i,... \}$.

Part (ii) follows directly from part (i) and Theorem \ref{thm-simple-integ-sl}. \end{proof}

\begin{rmk} Theorem 5 in \cite{FGM} shows that any simple  simple weight ${\mathcal D}(\infty)$-module with finite-dimensional weight spaces is multiplicity free. Moreover, \cite{FGM} provides an explicit description of such modules which is similar to the construction of the modules $X_{\mathfrak{sl}} (\mu)$. 
\end{rmk}

For a Borel subalgebra $\gb$ of $\gg$, we call a simple weight  $\gg$-module $M$ a $\gb$-\emph{pseudo highest weight module} if the root space $\gg^\alpha$ acts locally finitely on $M$ for every root $\alpha$ of $\gb$, and if $M$ is not a highest weight module. If $M$ is locally simple and is not a highest weight module, then $M$ is a $\gb$-pseudo highest weight module if and only if it is isomorphic to a direct limit of simple $(\gb \cap \gg_n)$-highest weight modules. The existence of simple pseudo highest weight modules is an effect related to the infinite-dimensionality of $\gg$, cf. \cite{DP}.

We now identify the highest and pseudo highest weight modules among the modules $X_{\mathfrak{sl}}(\boldsymbol{\mu})$.  Similarly to the finite-rank case, we first introduce some notation.  For a semi-infinite subset $A$ of $\Z_{>0}$ define $\varepsilon (A)\in {\mathbb C}^{{\mathbb Z}_{>0}}$ by $\varepsilon (A)_i = -1$ if $i \in A$ and $\varepsilon  (A)_j = 0$ if $j \notin A$. Furthermore, for $a \in {\mathbb C}$, $i_0 \in {\mathbb Z}_{>0}$ and $I \subset {\mathbb Z}_{>0}$ such that $i_0 \notin I$, define $\boldsymbol{\varepsilon} (i_0, a, I) \in \C^{\Z_{>0}}$ as follows: 
$$
\boldsymbol{\varepsilon} (i_0, a, I)_{i} := \begin{cases}
-1 & \mbox{ if } i \in I, \\  
a & \mbox{ if } i = i_0, \\  
0 & \mbox{ if } i \notin I. \\  
\end{cases} 
 $$
 Note that $\varepsilon (A) = \varepsilon (i_0,a,I)$ if and only if either $a=-1$, $A = I \cup \{ i_0\}$ or $a=0$, $A = I$. 
 
 Recall the definition of an $A$-compatible Borel subalgebra $\gb(\prec)$ from Subsection \ref{subsec-int-sim-sl}.

\begin{prop} \label{prop-hw-sl-infty} Let $\boldsymbol{\mu} \in {\mathbb C}^{{\mathbb Z}_{>0}}$ be such that ${\rm Int}^- (\mu) \neq \Z_{>0}$ and ${\rm Int}^+ (\mu) \neq  \Z_{>0}$,  and let $\gb = \gb (\prec)$ be a Borel subalgebra of  $\mathfrak{sl} (\infty)$. 

\begin{itemize}
\item[(i)]  $X_{\mathfrak{sl}} (\boldsymbol{\mu})$  is a $\mathfrak b$-highest weight module if and only if $\boldsymbol{\mu} \sim_{\mathfrak{sl}} \varepsilon (i_0,a,I)$ for some $a, I$, $i_0 \notin I$, such that $\gb$ is $I$- and $(I \cup \{ i_0\})$-compatible, or  $\boldsymbol{\mu} \sim_{\mathfrak{sl}} \varepsilon (A)$ for some semi-infinite  $A$ such that $\gb$ is $A$-compatible.
\item[(ii)]  $X_{\mathfrak{sl}} (\boldsymbol{\mu})$  is a $\mathfrak b$-pseudo highest weight module if and only if there are $j_0 \in \Z_{>0}$ and $J \subset \Z_{>0}$ ($j_0 \in J$ is allowed) such that $\gb$ is  $J$- and $(J \cup \{ j_0\})$-compatible and  $\mu_i \in \Z_{< 0}$ for $i \in J$, $\mu_j \in \Z_{\geq 0}$ for $j \notin J \cup \{ j_0 \}$, but $\boldsymbol{\mu} \not\sim_{\mathfrak{sl}} \varepsilon (j_0,a, J)$ and for any $a \in \C$.
\end{itemize}
\end{prop}
\begin{proof}
Since by Theorem \ref{thm-bounded-infinite-sl}(i), $X_{\mathfrak{sl}} (\boldsymbol{\mu})$ is a direct limit of the $\gg_n$-modules $X_{\mathfrak{sl}} (\boldsymbol{\mu}^{n+1})$, the statement for  pseudo highest weight modules follows from Proposition \ref{prop-sln-hw}. The statement for highest weight  modules follows by using again Proposition \ref{prop-sln-hw} and verifying when  the $(\gb \cap \gg_n)$-highest weight space of $X_{\mathfrak{sl}} (\boldsymbol{\mu}^{n+1})$ maps to the $(\gb \cap \gg_{n+1})$-highest weight space of $X_{\mathfrak{sl}} (\boldsymbol{\mu}^{n+2})$. 
\end{proof}

Using the above proposition we see that a nonintegrable highest weight module $X_{\mathfrak{sl}}(\boldsymbol{\mu})$ is one of the following two types.

\medskip
\noindent \emph{One-sided type.} This is the case when $\boldsymbol{\mu} = \boldsymbol{\varepsilon} (i_0, a, I)$ and $I$ or $\Z_{>0} \setminus I$ is finite. Assume that $\boldsymbol{\mu} = \boldsymbol{\varepsilon} (i_0, a, I)$, $I$ is finite, and $\gb = \gb(\prec)$ is a Borel subalgebra such that $X_{\mathfrak{sl}}(\boldsymbol{\mu})$ is $\gb$-highest weight module. Then one checks immediately that $X_{\mathfrak{sl}}(\boldsymbol{\mu})$ is also a $\gb'(\prec')$-highest weight module where $\prec'$ is a linear order on $\Z_{>0}$ isomorphic to the natural one and  such that $\gb'$ is $I$- and $(I \cup \{i_0\})$-compatible. This case corresponds to (23) in \cite{PP2}. The case when $\Z_{>0}\setminus I$ is finite corresponds to the case  (24) in \cite{PP2} and is related to (23) in \cite{PP2} via an outer automorphism of $\mathfrak{sl} (\infty)$.

\medskip
\noindent 
\emph{Two-sided type.} This is the case when $\boldsymbol{\mu} = \boldsymbol{\varepsilon} (i_0, a, I)$ and $I$ is semi-infinite,  or $\boldsymbol{\mu} = \boldsymbol{\varepsilon} (A) $ and $A$ is  semi-infinite. Under one of these assumptions, if $X_{\mathfrak{sl}}(\boldsymbol{\mu})$ is a highest weight module with respect to a Borel subalgebra $\gb (\prec)$ then  $X_{\mathfrak{sl}}(\boldsymbol{\mu})$ is also a highest weight module with respect to $ \gb' (\prec')$, where $\prec'$ is an order on $\Z_{>0}$ which is isomorphic to the natural order on $\Z$ and $\gb'$ is $I$- and $(I \cup \{i_0\})$-compatible, or, respectively, $A$-compatible. This case corresponds to (25) in \cite{PP2}.

\begin{examp}  
\begin{itemize}
\item[(i)] Define the following order $\prec$ on $\Z_{>0}$: $1 \prec 2 \prec 3 \prec x $ for any $x \geq 4$, and let  the order on $\Z_{>0}\setminus \{1,2,3 \}$ be isomorphic to the natural order on $\mathbb Q$. Set $I = \{ 1,2\}$, $i_0 = 3$, $a = \pi$. Then $X_{\mathfrak{sl}}(\boldsymbol{\varepsilon} (i_0, a, I))$ is $\gb(\prec)$-highest weight module of one-sided type, and  is a $\gb(<)$-highest weight module.
\item[(ii)] Define the following order on $\Z_{>0}$:
$$1 \prec 3 \prec 5 \prec \cdots \prec 6 \prec 4 \prec 2.$$
Let $A := \{1,3,5,... \}$. Then   $X_{\mathfrak{sl}}(\boldsymbol{\varepsilon} (A))$ is a $\gb(\prec)$-highest weight module of two-sided type.
\item[(iii)] The module $X_{\mathfrak{sl}}(1,2,\sqrt{2}, -1,-2,-3,...)$ is a $\gb (<)$-pseudo highest weight module.
\end{itemize}
\end{examp}

\subsubsection{The case  $\gg = \mathfrak{sp} (\infty)$}

\begin{thm} \label{thm-bounded-infinite-sp}
\begin{itemize} 
\item[(i)]  Every simple bounded nonintegrable $\mathfrak{sp} (\infty)$-module  $M$ is isomorphic  to $X_{\mathfrak{sp}}(\boldsymbol{\mu})$ for some $\mu \in \C^{{\mathbb Z}_{>0}}$. In particular,  $X_{\mathfrak{sp}}(\boldsymbol{\mu}) = \underrightarrow\lim X_{\mathfrak{sp}}(\boldsymbol{\mu}^n)$ where $\boldsymbol{\mu}^n = (\mu_1,...,\mu_n)$.
\item[(ii)] 
$X_{\mathfrak{sp}}(\boldsymbol{\mu}) \simeq X_{\mathfrak{sp}}(\boldsymbol{\mu}')$  if and only if $\boldsymbol{\mu} \sim_{\mathfrak{sp}} \boldsymbol{\mu'}$.
\item[(iii)] The root space $\gg^{\alpha}$ acts locally finitely on $X_{\mathfrak{sp}}(\mu)$ if and only if 
$$\alpha \in \{\pm \varepsilon_i - \epsilon_j, - 2 \varepsilon_j,  \varepsilon_k \pm  \epsilon_\ell, 2 \varepsilon_k  \; | \: j  \in  {\rm Int}^+ (\mu),  k  \in  {\rm Int}^- (\mu)\}.
$$
In particular, $X_{\mathfrak{sp}}(\mu)$ is always nonintegrable and is cuspidal if and only if $ {\rm Int} (\mu) = \emptyset$; in the latter case $X_{\mathfrak{sp}}(\mu) = {\mathcal F}_{\mathfrak{sp}} (\mu)$.

\item[(iv)] $\Supp X_{\mathfrak{sp}}(\boldsymbol{\mu}) = \{\lambda +\left( \frac{1}{2}\right)^{(\infty)} \; | \; \boldsymbol{\lambda} \sim_{\mathfrak{sp}}  \boldsymbol{\mu} \}$. In particular, $\Supp X_{\mathfrak{sp}}(\boldsymbol{\mu})$ determines $X_{\mathfrak{sp}}(\boldsymbol{\mu})$ up to isomorphism.
\end{itemize}
\end{thm}
\begin{proof} The proof is analogous to the proof of Theorem \ref{thm-bounded-infinite-sl}. It is enough to prove (i). Let $\lambda \in \Supp M$ and $m \in M^{\lambda}$ be such that $M_n = U(\gg_n) \cdot m$ is simple for $n>N$. Such $\lambda$ and $m$ exist thanks to Corollary \ref{cor-sp-deg1}. Set $\mu_i = \lambda_i - \frac{1}{2}$, $\boldsymbol{\mu}^n = (\mu_1,...,\mu_n)$, and $\boldsymbol{\mu}= (\mu_1,\mu_2,...)$. Then by Theorem \ref{thm-x-sp-fin} and Proposition \ref{thm-bounded-finite-sp}(ii), $M_n \simeq X_{\mathfrak{sp}} (\boldsymbol{\mu}^n)$. Recall the notation introduced prior to Lemma \ref{prop-restr-sp2n}. In particular, $\widehat{\gg}_n \simeq \gg_n \oplus \C$ is a subalgebra of $\gg_{n+1}$ and $\widehat{M}_n = U(\widehat{\gg}_n) \cdot m$ is a simple $\widehat{\gg}_n$-module isomorphic to 
$X_{\widehat{\gg}_n} (\boldsymbol{\mu}^n ; \mu_{n+1})$. We similarly have $\widehat{M}_{n+1} \simeq X_{\widehat{\gg}_{n+1}} (\boldsymbol{\mu}^{n+1}; \mu_{n+2})$. By Lemma \ref{prop-restr-sp2n}, there is a unique $\widehat{\gg}_n$-monomorphism $X_{\mathfrak{sp}} (\boldsymbol{\mu}^n; \mu_{n+1}) \to X_{\mathfrak{sp}} (\boldsymbol{\mu}^{n+1}; \mu_{n+2})$. Since $\gg =  \underrightarrow\lim \widehat{\gg}_n$, we have $X_{\mathfrak{sp}}(\boldsymbol{\mu})  =  \underrightarrow\lim X_{\mathfrak{sp}}(\mu^n) =   \underrightarrow\lim X_{\widehat{\gg}_n}(\mu^{n}; \mu_{n+1})$.  Lastly, we verify that if $M' = X_{\mathfrak{sp}}(\boldsymbol{\mu})$, then for $\lambda' \in \Supp M'$, $m' \in (M')^{\lambda'}$, and $M'_n = U(\gg_n)\cdot m'$, then $M_n' \simeq X_{\mathfrak{sp}} (\mu_1,...,\mu_n)$. Hence, $M \simeq M'$.\end{proof}

\begin{cor} Every simple nontrivial bounded $\mathfrak{sp} (\infty)$-module is isomorphic to $X_{\mathfrak{sp}} (\boldsymbol{\mu})$ for some $\boldsymbol{\mu} \in \C^{\Z_{>0}}$ or to $V$.
\end{cor}
\begin{proof}
The statement follows from Theorem \ref{thm-bounded-infinite-sp}(i) and Proposition \ref{pro2}.
\end{proof}

In order to identify the highest and pseudo highest weight modules of the form $X_{\mathfrak{sp}}(\boldsymbol{\mu})$, we introduce some notation.  To each Borel subalgebra $\gb  (\prec, \sigma)$ of $\fg$ we assign integer sequences $\omega_{\gb}$ and $\delta_{\gb}$ as follows. For $i \in {\mathbb Z}_{>0}$, we set $(\omega_{\gb})_i := 0$ if $\sigma(i) = -1$ and $(\omega_{\gb})_i := - 1$ if $\sigma(i) = +1$. Furthermore, if $\prec$ has a maximal element $j_0$, then $\delta_{\gb} := -\varepsilon_{j_0}$ if $\sigma(j_0) = +1$ and $\delta_{\gb} := \varepsilon_{j_0}$ if $\sigma(j_0) = -1$, where $(\varepsilon_{j_0})_i  = \delta_{j_0i}$. Recall that $\left(\frac{1}{2}\right)^{(\infty)} := \left(\frac{1}{2}, \frac{1}{2}, ...\right)$.

\begin{prop} \label{prop-hw-sp-infty} Let $\boldsymbol{\mu} \in {\mathbb C}^{{\mathbb Z}_{>0}}$ and $\gb = \gb (\prec, \sigma)$ be a Borel subalgebra of  $\mathfrak{sp} (\infty)$. 

\begin{itemize}
\item[(i)] If $\prec$ has a maximal element $j_0$, then $X_{\mathfrak{sp}} (\boldsymbol{\mu})$ is a $\mathfrak b$-highest weight module if and only if $\boldsymbol{\mu} \sim_{\mathfrak{sp}} \omega_{\gb}$  or $\boldsymbol{\mu} \sim_{\mathfrak{sp}} (\omega_{\gb} + \delta_{\gb} )$ and in this case $\omega_{\gb}+ \left(\frac{1}{2}\right)^{(\infty)} $ and $\omega_{\gb} + \left(\frac{1}{2}\right)^{(\infty)}  + \delta_{\gb}$, respectively, are the $\gb$-highest weights of $X_{\mathfrak{sp}}(\boldsymbol{\mu})$. 

\item[(ii)]  If $\prec$ has no maximal element, then $X_{\mathfrak{sp}} (\boldsymbol{\mu})$ is a $\mathfrak b$-highest weight module if and only if $\boldsymbol{\mu} \sim_{\mathfrak{sp}} \omega_{\gb}$  and in this case $\omega_{\gb}+ \left(\frac{1}{2}\right)^{(\infty)} $  is the $\gb$-highest weight of $X_{\mathfrak{sp}} (\boldsymbol{\mu})$.

\item[(iii)] $X_{\mathfrak{sp}} (\boldsymbol{\mu})$ is a $\mathfrak b$-pseudo highest weight module if and only if $\boldsymbol{\mu} \not\sim_{\mathfrak{sp}} \omega_{\gb}$,  $\boldsymbol{\mu} \not\sim_{\mathfrak{sp}} (\omega_{\gb} + \delta_{\gb} )$, $\mu_i \in \Z_{\geq 0}$ whenever $\sigma(i) = -1$, and $\mu_j \in \Z_{<0}$ whenever $\sigma(j) = +1$. 
\end{itemize}
\end{prop}
\begin{proof}
We use the same reasoning as in the proof of Proposition \ref{prop-hw-sl-infty}. The statements follows from Theorem \ref{thm-bounded-infinite-sp}(i) and Proposition \ref{prop-sp2n-hw}. 
\end{proof}

All highest weight modules $X_{\mathfrak{sp}}(\boldsymbol{\mu})$ are described in the following corollary.
\begin{cor} \label{cor-class-hw-sp} Let $\boldsymbol{\mu} \in {\mathbb C}^{{\mathbb Z}_{>0}}$. 
\begin{itemize}
\item[(i)] The module  $X_{\mathfrak{sp}}(\boldsymbol{\mu})$  is a $\mathfrak b$-highest weight module for some Borel subalgebra  $\gb = \gb (\prec, \sigma)$ if and only if  there are $A \subset {\mathbb Z}_{>0}$, $a \in A$, $b \notin A$, so that $\mu  \sim_{\mathfrak{sp}} \omega_A$ or $\mu  \sim_{\mathfrak{sp}} \omega_A + \varepsilon_a$ or $\mu  \sim_{\mathfrak{sp}} \omega_A -\varepsilon_b$, where  $\omega_A$ is defined as follows: $(\omega_A)_i = \frac{1}{2}$ if $i \in A$ and  
$(\omega_A)_j = -\frac{1}{2}$ if $j \notin A$.
\item[(ii)]  The module  $X_{\mathfrak{sp}}(\omega_A)$ is a $\mathfrak b (\prec, \sigma )$-highest weight module if and only if $\sigma(i) =-1$ for $i \in A$ and $\sigma(j) =+1$ for $j \notin A$.  The module   $X_{\mathfrak{sp}}(\omega_A + \varepsilon_a)$ (respectively, $X_{\mathfrak{sp}}(\omega_A  - \varepsilon_b)$ ) is a $\mathfrak b (\prec, \sigma )$-highest weight module if and only if $\sigma(i) =-1$ for $i \in A$ and $\sigma(j) =+1$ for $j \notin A$, and $a$ (respectively, $b$) is a maximal element of $\prec$.
\item[(iii)] All simple nonitegrable bounded highest weight $\gg$-modules are obtained form $X_{\mathfrak{sp}} (0,0,0,...)$ and $X_{\mathfrak{sp}} (1,0,0,...)$ via a twist by an automorphism of $\gg$ that fixes $\gh$.
\end{itemize}
\end{cor}
\begin{examp}
The module $X_{\mathfrak{sp}} (1^{(\infty)})$ is a $\gb$-pseudo highest weight module for any $\gb = \gb (\prec, \sigma)$ such that $\sigma (i) = -1$ for all $i \in \Z_{>0}$.
\end{examp}

\section{Annihilators of simple bounded weight $\gs\gl(\infty)$-, $\go(\infty)$-, $\gs\gp(\infty)$-modules}

In this section,  $\Ann(\cdot)=\Ann_{U(\fg)}(\cdot)$ for $\fg=\gs\gl(\infty)$, $\go(\infty)$ or $\gs\gp(\infty)$.

\subsection{Annihilators of simple bounded nonintegrable modules of $\mathfrak{sl} (\infty)$}

We start by recalling the classification of the primitive  ideals of  $U(\mathfrak{sl} (\infty))$ obtained in \cite{PP}. For $x,y \in \Z_{\geq 0}$ and partitions $\lambda, \mu$, denote by $I(x,y,\lambda, \mu)$ the annihilator of the  $U(\mathfrak{sl} (\infty))$-module $(S^{\cdot} (V))^{\otimes x} \otimes (\Lambda^{\cdot}(V))^{\otimes y} \otimes S^{\lambda} V \otimes S^{\mu} V_*$.

\begin{thm}[Theorem 2.1, \cite{PP}]
All ideals  $I(x,y,\lambda, \mu)$ are primitive and nonzero, and any nonzero primitive
ideal $I$ of $U(\mathfrak{sl} (\infty))$  equals exactly one of these ideals.
\end{thm}

The goal of this subsection is to prove the following.

\begin{thm} \label{thm-ann-nonint-sl}
Let $M$ be a simple bounded nonintegrable module of  $ \gg=\mathfrak{sl} (\infty)$. Then $\Ann M = I(1,0,\emptyset,\emptyset)$.
\end{thm}
\begin{proof}
Set $J:=I(1,0,\emptyset,\emptyset)$. Using Theorem \ref{thm-bounded-infinite-sl}, let $M \simeq X_{\mathfrak{sl}}(\mu)$ for some $\mu \in {\mathbb C}^{{\mathbb Z}_{\geq 0}}$.  Then all simple subquotients of $M_{|\fg_n}$ are simple multiplicity-free $\fg_n$-modules, and by Theorem \ref{thm-clas-bounded} are isomorphic to twisted localizations of simple multiplicity-free highest weight $\fg_n$-modules. Since twisted localization does not change annihilators by Lemma \ref{lem-loc-ann}, we conclude that the annihilators in $U_n$ of the simple constituents of $M_{|\fg_n}$ are annihilators of simple multiplicity-free highest weight modules.  Primitive ideals of $U_n$ are invariant under conjugation by inner automorphisms of $\fg_n$, so the annihilators in question are annihilators of simple highest weight modules with respect to the Borel subalgebra $\fb_n$.  Next, by Proposition \ref{prop-sln-hw}, a simple multiplicity-free $\fb_n$-highest weight $\fg_n$-module has highest weight $\lambda=\varepsilon_{\gb_n} (i_0,a)$
where $a\in\mathbb{C}$ is arbitrary  and $1\leq i_0\leq n+1$.

Let $\tilde\lambda$ be the weight of $\mathfrak{sl}(\infty)$, which extends $\lambda$ by zero (i.e. $\tilde\lambda_j=\lambda_j$ for $1\leq j\leq n+1$, $\tilde\lambda_j=0$ for $j>n+1$), and let $L(\tilde\lambda)$ be the simple highest weight $\mathfrak{sl}(\infty)$-module with highest weight $\tilde\lambda$ with respect to the Borel subalgebra $\fb(<)$ of $\mathfrak{sl}(\infty)$; here $<$ is the usual order on the $\mathbb{Z}_{>0}$ and $\fb(<)=\underrightarrow\lim\fb_n$. By Example 8.1 in \cite{PP2}, the  $\mathfrak{sl} (\infty)$-module $L(\tilde\lambda)$ has annihilator $J$, unless $L(\tilde\lambda)$ is isomorphic to $S^t (V)$ (in the latter case, $\tilde\lambda=(t,0\dots)$ ). Moreover, $ \Ann S^t (V)=I(0, 0, (t), \emptyset)$ contains $J$ by Theorem \ref{thm-simple-integ-sl} in \cite{PP}. Since $L_{\fg_n}(\lambda)$ is a $\fg_n$-submodule of $L(\tilde\lambda)$, we have

\begin{equation}\label{eq-inter}
 J \cap U_n \subset \bigcap \Ann_{U_n} L,
\end{equation}
the intersection being taken over all annihilators of simple multiplicity-free $\fb_n$-highest weight $\fg_n$-modules $L$.
Thus $J \subset \Ann M$. 

Now, Theorem 5.3 in \cite{PP} implies that $J$ is contained properly only in primitive ideals of the form $I(0,0,\emptyset, \mu')$, or  $ I(0,0,\lambda', \emptyset)$ where $\lambda'$ and $\mu'$ are partitions. However, the ideals $I(0,0,\lambda', \emptyset) \cap U_n$ and $I(0,0,\emptyset, \mu') \cap U_n$  have finite codimension in $U_n$ for each $n$, therefore  cannot annihilate infinite-dimensional modules. Since $M_{|\fg_n}$ has at least one infinite-dimensional simple constituent, we conclude that 
$$\Ann M = J. $$ \end{proof}

\subsection{Annihilators of simple bounded nonintegrable modules of $\mathfrak{sp} (\infty)$}
In this subsection $\gg = \mathfrak{sp} (\infty)$. Recall that the two Shale-Weil $\fg_n=\gs\gp(2n)$-modules $L\left(\frac{1}{2}, \frac{1}{2},\dots, \frac{1}{2}\right)$ and $L\left( \frac{3}{2},  \frac{1}{2}, \dots,  \frac{1}{2}\right)$ have the same annihilator $J_n\subset U_n$. Since the $\fg$-module $L\left(\frac{1}{2}, \frac{1}{2},\dots \right)$ decomposes as a direct sum of copies of $L\left(\frac{1}{2}, \frac{1}{2},\dots, \frac{1}{2}\right)$ and $L\left( \frac{3}{2},  \frac{1}{2}, \dots,  \frac{1}{2}\right)$ after restriction to $\fg_n$, we conclude that $I_{\rm sw}:=\underrightarrow\lim\, J_n$ is a well-defined  primitive ideal of $U(\fg)$.   Clearly $I_{\rm sw}=\Ann L\left(\frac{1}{2}, \frac{1}{2},\dots, \frac{1}{2} \right)$. Note also that the intersection $I_{\rm sw}\cap U_n$ is a maximal Joseph ideal for each $n$, and therefore $I_{\rm sw}$ is a maximal ideal in $U(\fg)$. 

\begin{thm}
Let $M$ be a simple bounded nonintegrable module of  $\gg =  \mathfrak{sp} (\infty)$. Then $\Ann M = I_{\rm sw}$.
\end{thm}
\begin{proof} Using twisted localization,Theorem \ref{thm-clas-bounded} and Lemma \ref{lem-loc-ann}, we conclude that all simple multiplicity free  $\gg_n$-modules $M_n$ have   annihilator equal to $J_n$.  Therefore  $\Ann M= I_{\rm sw}$ for any nonitegrable simple multiplicity-free weight $\fg$-module $M$.  \end{proof}

\subsection{Annihilators of simple bounded integrable weight $\mathfrak{sl} (\infty)$-, $\mathfrak{o} (\infty)$- , $\mathfrak{sp} (\infty)$-modules}

\begin{thm} Let $\gg = \mathfrak{sl} (\infty), \mathfrak{o} (\infty), \mathfrak{sp} (\infty)$, and let 
 $M$ be a nontrivial simple bounded integrable $\gg$-module.

\begin{itemize}
\item[(i)] If $\gg = \mathfrak{sl} (\infty)$, then
\begin{itemize}
\item[(a)] $\Ann M \simeq I(0,1,\emptyset, \emptyset)$ for $M = \Lambda_A^{\frac{\infty}{2}}  V$;
\item[(b)] $\Ann M \simeq I(1,0,\emptyset, \emptyset)$ for $M = S_A^{\infty}  V$ or $M = S_A^{\infty} V_*$;
\item[(c)] $\Ann M \simeq I(0,0,\lambda, \emptyset)$ for $M = S^{\lambda} V$;
\item[(d)] $\Ann M \simeq I(0,0,\emptyset, \mu)$ for $M = S^{\mu} V_*$.
\end{itemize}
\item[(ii)] If $\gg = \mathfrak{sp} (\infty)$, then $M \simeq  V$ and 
$ \Ann  V\neq I_{\rm sw}$. 
\item[(iii)] If $\gg = \mathfrak{o} (\infty)$ then,  under the assumption that $M\not\simeq V$, $\Ann M$ does not depend on $M$. Moreover $\Ann M\neq 0$ and $\Ann M\neq \Ann V$.
\end{itemize}
\begin{proof}
Claims  (i)(c) and (i)(d) follow directly from the definitions. For (i)(a) one needs to check that $\Ann \Lambda^{\frac{\infty}{2}}_A(V)$ coincides with $\Ann \Lambda^\cdot(V)$ for any $A$. This can be done by checking that ($\Ann \Lambda^{\frac{\infty}{2}}_A(V))\cap U_n$ and ($\Ann \Lambda^{\cdot}(V))\cap U_n$ coincide for any $M$. Similarly, we prove (b) by checking that ($\Ann S_A^{\infty}  V)\cap U_n=(\Ann S^{\cdot}(V))\cap U_n$ and ($\Ann S_A^{\infty}  V_*)\cap U_n=(\Ann S^{\cdot}(V_*))\cap U_n$  for any $n$. Claims (ii) and (iii) are also straightforward. In particular, it is easy to check that for $\Ann M$ in claim (iii) the intersection $(\Ann M)\cap U_n$   for $\fg_n\simeq \go(2n+1)$ equals $\Ann_{U_n}\mathcal{S}_n$ (for $\fg_n=\go(2n)$ we have $(\Ann M)\cap U_n=(\Ann_{U_n}\mathcal{S}^+_n)\cap(\Ann_{U_n}\mathcal{S}^-_n$)).
\end{proof}

\end{thm}

Note that Theorem 7.3, together with Theorem 7.4(i)(b), shows that $\Ann X_{\mathfrak{sl}}(\mu)= I(1,0,\emptyset, \emptyset)$ for any $\mu$ such that $X_{\mathfrak{sl}}(\mu)\not\simeq\mathbb{C}$.

\subsection{A final remark}

Note that all claims made in this paper about isomorphisms of weight $\mathfrak{sl}(\infty)$-, $\mathfrak{o}(\infty)$-, $\mathfrak{sp}(\infty)$-modules concern weight modules for a fixed Cartan subalgebra $\fh$. If one fixes two Cartan subalgebras $\fh_1$ and $\fh_2$ of $\fg$ and asks whether an $\fh_1$-weight module  $M_1$ is isomorphic to an $\fh_2$-weight module $M_2$, we do not have an answer to this question. It is however obvious, that a $\fg$-isomorphism
  $M_1\simeq M_2$ implies $\Ann M_1=\Ann M_2$ and that $\fh_2$ acts locally finitely on $M_1$, and vice versa. This allows to rule out some possible isomorphisms but the general isomorphism problem appears to be open.

\bibliography{ref,outref,mathsci}
\def\cprime{$'$} \def\cprime{$'$} \def\cprime{$'$} \def\cprime{$'$}
  \def\cprime{$'$}
\providecommand{\bysame}{\leavevmode\hbox
to3em{\hrulefill}\thinspace}
\providecommand{\MR}{\relax\ifhmode\unskip\space\fi MR }
\providecommand{\MRhref}[2]{%
  \href{http://www.ams.org/mathscinet-getitem?mr=#1}{#2}
} \providecommand{\href}[2]{#2}

\end{document}